\documentclass[a4paper,11pt]{article}
\usepackage{amsfonts}
\usepackage[T1]{fontenc}
\usepackage{microtype}
\usepackage{makeidx}
\usepackage{setspace}
\usepackage{faktor}
\usepackage{authblk}
\usepackage{graphicx}
\usepackage[all]{xy}
\usepackage{amsmath}
\usepackage{amssymb}
\usepackage{booktabs}
\usepackage{braket}
\usepackage{array}
\usepackage{tabularx}
\usepackage{multirow}
\usepackage{indentfirst}
\usepackage{cite}
\usepackage{stmaryrd}
\usepackage{faktor}
\usepackage{titling}
\usepackage{tikz}
\usepackage{dsfont}
\usepackage{url}
\usepackage{hyperref}
\usepackage{tikz-cd}
\usetikzlibrary{matrix,arrows,decorations.pathmorphing}
    \renewcommand{\leq}{\leqslant}
    
\usepackage{amsthm}
\theoremstyle{plain}
\newtheorem{thm}{Theorem}[section]
\newtheorem{dfn}[thm]{Definition}

\newtheorem{prop}[thm]{Proposition}
\newtheorem{cor}[thm]{Corollary}
\newtheorem{conge}[thm]{Conjecture}
\newtheorem{prob}[thm]{Problem}

\theoremstyle{remark}
\newtheorem{oss}[thm]{Remark}

\theoremstyle{definition}
\newtheorem{ex}[thm]{Example}

\DeclareMathOperator{\spn}{span}

\DeclareMathOperator{\gl}{\mathrm{GL}}

\DeclareMathOperator{\gr}{\mathrm{Gr}}

\DeclareMathOperator{\id}{\mathrm{Id}}
\DeclareMathOperator{\mat}{\mathrm{Mat}}

\DeclareMathOperator{\imm}{\mathrm{Im}}

\DeclareMathOperator{\F}{\mathbb{C}}
\DeclareMathOperator{\JP}{\mathcal{J}(P)}
\DeclareMathOperator{\End}{\mathrm{End}}
\DeclareMathOperator{\Hom}{\mathrm{Hom}}

\DeclareMathOperator{\PP}{\mathbb{P}}
\DeclareMathOperator{\C}{\mathbb{C}}
\DeclareMathOperator{\Q}{\mathbb{Q}}
\DeclareMathOperator{\N}{\mathbb{N}}
\DeclareMathOperator{\Z}{\mathbb{Z}}
\DeclareMathOperator{\R}{\mathbb{R}}

\DeclareMathOperator{\lex}{\mathrm{lex}}

\DeclareMathOperator{\supp}{\mathrm{supp}}
\DeclareMathOperator{\var}{\vartriangleleft}

\DeclareMathOperator{\seg}{\mathrm{Seg}}

\title{\bf{Immanant varieties}}
\author{}
\author{Davide Bolognini\thanks{Dipartimento di Ingegneria Industriale e Scienze Matematiche, Università Politecnica delle Marche, Ancona, Italy.
\href{mailto:davide.bolognini.cast@gmail.com}{davide.bolognini.cast@gmail.com} } \ and Paolo Sentinelli\thanks{ Dipartimento di Matematica, Politecnico di Milano, Milan, Italy. \\ \href{mailto:paolosentinelli@gmail.com}{paolosentinelli@gmail.com}}}
\date{}
\begin{document}

\maketitle

\vspace{-4em}

\begin{abstract}
We introduce immanant varieties, associated to simple characters of a finite group. They include well-studied classes of varieties, as Segre embeddings, Grassmannians and certain other classes of Chow varieties. For a one-dimensional character $\chi$, we define $\chi$-matroids by a maximality property. For trivial characters, by exploring the combinatorics of incidence stratifications, we provide a set of generators for the Chow vector spaces of the corresponding immanant varieties. \end{abstract}


$\,$

\textbf{Math. sub class:} \emph{05B35, 05E45, 05E18, 06A07, 15A69, 20B30}

\section{Introduction}
The set $\seg(k,n)$ of rank-one tensors in $\PP\left(\bigotimes_{i=1}^k\C^n\right)$ is a projective variety; it is the image under the Segre embedding of a Cartesian product of projective spaces.
Given a non-zero endomorphism $f$ of the vector space $\bigotimes_{i=1}^k\C^n$, the Zariski closure of the set $f(\seg(k,n))$ is an irreducible algebraic set in $\PP\left(\bigotimes_{i=1}^k\C^n\right)$, i.e. a projective variety (Definition \ref{definizione} and Proposition \ref{teorema irriducibile}).
Under the standard action of the symmetric group $S_k$, the tensor product $\bigotimes_{i=1}^k\C^n$ has the structure of a module over the group algebra $\C[S_k]$. The induced representations of simple representations of a subgroup $G \subseteq S_k$ provide idempotents of $\C[S_k]$ which are endomorphism of the tensor product above. The image of such endomorphisms are the so-called \emph{symmetry classes of tensors} (see e.g. \cite{Merris} for a general theory, \cite{portoghesi} for a survey, and \cite{marcus} for generalized matrix functions).

This construction allows us to associate to each simple character $\chi$ of $G \subseteq S_k$ a projective variety $\gr_{\chi}(k,n)$ which we call \emph{immanant variety} (Definition \ref{definizione varieta immananti}), since for $G=S_k$ the parametric equations defining it are written in terms of immanants (Theorem \ref{equazioni}), which are generalizations of the determinant of a square matrix (see e.g. \cite{stembridge} and reference therein). In our context we use in a wider sense the word \emph{immanant} for a matrix (see Definition \ref{def immanant}), including the standard notions of immanants and generalized matrix functions for simple characters of finite groups (see \cite{marcus}).

Among immanant varieties there is a famous one, namely the complex Grassmannian; it arises by considering the alternating character of a symmetric group and the immanant involved in the parametric equations is the determinant.
In this article we prove results recovering part of the geometrical and combinatorial richness of Grassmannians for a wide class of immanant varieties. The main features explored are the following ones.

$\,$

  \textbf{$\chi$-matroids:} it is well-known that the points of a Grassmannian
  $\gr_{\C}(k,n)$ correspond to rank $k$ matroids on the ground set $[n]$, representable over $\C$. A characterization of matroids, due to Gale, is by their maximality property (see \cite[Theorem 1.3.1]{coxeter matroids}). A maximality property can be defined in the more general case of one-dimensional characters $\chi$ of any finite group, leading to the definition of {\em $\chi$-matroid}, see Definition \ref{chi matroidi}. Although the points of $\gr_{\chi}(k,n)$ are not $\chi$-matroids in general (Example \ref{non e chi matroide}), the points of the varieties associated to trivial characters are $\chi$-matroids (Corollary \ref{corollario massimalita}).

    $\,$

  \textbf{Incidence stratifications:} the notion of \emph{incidence stratification}
  has been introduced by the authors in \cite{BoloSenti}. The stratification of a Grassmannian variety by its Schubert varieties is an example of incidence stratification (see \cite[Proposition 4.16]{BoloSenti}). This construction provides a correspondence between  Schubert varieties and principal order ideals of the Bruhat order on Grassmannian permutations, thanks to the maximality property of matroids.  In the same vein, for the trivial character $\mathbf{1}_G:G \rightarrow \{1\}$, the maximality property of $\gr_{\mathbf{1}_G}(k,n)$ guarantees the existence of an incidence stratification whose strata are projective varieties (Proposition \ref{teorema ideali ordine principali} and Corollary \ref{corollario Schubert irriducibili}) and
  whose inclusion poset is graded by dimension (Theorem \ref{teorema ciao}) and rank-symmetric (Corollary \ref{specchio}).
  These stratifications are closely related to \emph{Seshadri stratifications}, as recently introduced by
  Chirivì, Fang and Littelmann in \cite{chirivi}. In fact, all the axioms defining Seshadri stratifications are satisfied by our stratifications, except possibly smoothness in codimension one (Example \ref{osservazione i}).

     $\,$

   \textbf{Chow vector spaces:} it is well-known that the Chow group of the Grassmannian is free, with a basis given by the classes of Schubert varieties (see e.g. \cite[Section 14.7]{Fulton}). By using the combinatorial results proved along the paper, we find a set of generators for the vector space obtained tensoring by $\mathbb{Q}$ the Chow group of $\gr_{\mathbf{1}_G}(k,n)$ (Theorem \ref{teorema ciao}), namely the set of rational equivalence classes of the strata of the incidence stratification explained above. Moreover, applying Proposition \ref{polia}, we give an upper bound for the Hilbert-Poincar\'e polynomial of the Chow vector space.

$\,$

The last part of the paper is devoted to some conjectures and open problems. One of them concerns shellability of intervals in posets arising from the geometry of $\gr_{\mathbf{1}_G}(k,n)$. In parabolic quotients of Coxeter groups, the order complexes of Bruhat intervals are shellable (see \cite[Theorem 2.7.5]{BB}); in particular, this holds for the Bruhat order of Grassmannian permutations, which is the inclusion poset of Schubert varieties in a Grassmannian. Since the latter is a distributive lattice, the shellability of its intervals can be deduced from a general result of Bj\"{o}rner \cite{biorner}. The same can be easily proved for the intervals of the inclusion poset of the incidence strata in $\gr_{\mathbf{1}_{S_k}}(k,n)$ (see the end of Section \ref{sezione triviale}). Despite the fact that, for arbitrary groups $G \subseteq S_k$, this poset is not a lattice in general (see Example \ref{esempio ciclico}), several experiments led us to conjecture that it is shellable (Conjecture \ref{conge shell}).

\section{Notation and preliminaries} \label{sezione preliminari}

In this section we fix notation and recall some definitions useful for the rest of the paper. We refer to \cite{Stanley} for posets and their incidence algebras, to \cite[Chapter XVIII]{lang} for the representation theory of finite groups, to \cite{coxeter matroids} and \cite[Section 2.3]{BoloSenti} for matroids, to \cite{Landsberg} and \cite{Merris} for tensors, and to \cite{Fulton} for intersection theory.

Let $\Z$ be the ring of integer numbers, $\Q$ the field of rational numbers, $\R$  the field of real numbers, $\C$ the field of complex numbers and $\mathbb{N}$ the set of positive integers. For $n\in \mathbb{N}$, we use the notation $[n]:=\left\{1,2,\ldots,n\right\}$. For a finite set $X$, we denote by $|X|$ its cardinality, by $\mathcal{P}(X)$ its power set, by $X^n$ or $X^{\times n}$ its $n$-th power under Cartesian product. If $x\in X^n$, we denote by $x_i$ the projection of $x$ on the $i$-th factor.
If $f : X \rightarrow Y$ is a function, we let $\imm(f):=\{f(x):x \in X\}$. We denote by $f$ also the induced function $f: \mathcal{P}(X) \rightarrow \mathcal{P}(Y)$.

If $(X,\preccurlyeq)$ is a poset, then $X^n$ is the poset given by letting $x \leqslant y$ if and only $x_i \preccurlyeq y_i$, for all $i\in [n]$ and $x,y\in X^n$. The set $[n]$ is a poset under the natural order; so, for $k \in \N$, the set $[n]^k$ is considered to be a poset. We denote by $\var$ a covering relation in a poset $P$, i.e. $x \var y$ if and only if $x<y$ and $\{z\in P: x < z < y\}=\varnothing$. In the category of graded posets, a morphism $f : X \rightarrow Y$
is an order preserving function such that $\rho_2(f(x))=\rho_1(x)$,
for all $x\in X$, where $\rho_1$ is the rank function of $X$ and $\rho_2$ is the rank function of $Y$.

Let $O, O_1$ and $O_2$ be objects in a category. The notation $\Hom(O_1,O_2)$ stands for the set of morphisms between $O_1$ and $O_2$. We let $\End(O):=\Hom(O,O)$ and $O_1 \simeq O_2$ denotes the existence of an isomorphism.

 Let $n \in \N$ and $V$ be an $n$-dimensional $\C$-vector space; define an equivalence relation $\sim$ on $V\setminus \{0\}$ by setting $u \sim v$
if and only if $\dim \left( \spn_{\C}\{u,v\}\right)=1$, for all $u,v \in V\setminus \{0\}$. Then, for any subset $X\subseteq V$, we let $\PP(X):= \pi(X\setminus \{0\})$, where
$\pi : V\setminus \{0\} \rightarrow \left(V\setminus \{0\}\right)/\sim$
is the canonical projection. In particular, $\PP(V)$ is the projective space of $V$. For $v\in V\setminus \{0\}$, we let $[v]:=\pi(v) \in \PP(V)$.
Let $k \in \N$; with $V^{\otimes k}$ we denote the $k$-th tensor power
 of $V$. We let $\mathrm{seg}^{k,n}  : \PP(V)^{\times k} \rightarrow \PP\left(V^{\otimes k}\right)$ be the function defined by
 $$\mathrm{seg}^{k,n} ([v_1],\ldots,[v_k])=[v_1\otimes \ldots \otimes v_k],$$
 for all $v_1,\ldots,v_k \in V\setminus \{0\}$. This is the so-called \emph{Segre embedding} and we set $\seg(k,n):=\imm(\mathrm{seg}^{k,n})$. It is well-known that $\seg(k,n)$
 is a $k(n-1)$-dimensional projective subvariety of $ \PP\left(V^{\otimes k}\right)$.

We end this section by recalling the definition of incidence stratification of a projective set, as appears in \cite[Section 4]{BoloSenti}.
Let $P=\left([n],\leqslant_P\right)$ be a poset of cardinality $n$.
An {\em order ideal} of $P$ is a subset $I \subseteq P$ such that $i \in I$ and $j \leqslant_P i$ imply $j \in I$. The distributive lattice of order ideals of a poset $P$ is denoted by $\JP$. It is clear that there is a bijection between $\JP$ and the sets $\left\{\max(I):I\in \JP\right\}$.
For $x\in P$, we define the {\em principal order ideal} generated by $x$ by setting $$x^\downarrow:=\left\{y\in P: y\leqslant_P x\right\}.$$

For $h,k \in \N$ we denote by $\mathrm{Mat}_{h,k}(\C)$ the algebra of matrices whose entries are complex numbers. If $A \in \mathrm{Mat}_{h,k}(\C)$, $i\in [h]$ and $j \in [k]$, $A_{i,j}$ is the entry in position $(i,j)$ of the matrix $A$.
\begin{dfn}
The {\em incidence algebra} of $P$ over $\C$ is $$I(P;\C):=\left\{A \in \mathrm{Mat}_{n,n}(\C): i \nleqslant_P j \Rightarrow A_{i,j}=0 \right\}.$$
The \emph{incidence group} $I^*(P;\C)$ of $P$ over $\C$ is the group of invertible elements of $I(P;\C)$. \end{dfn}

The subalgebra $I(P;\C) \subseteq \End(\C^n)$ has invariant-subspace lattice isomorphic to $\mathcal{J}(P)$, where
$I(P;\F)$ acts on the elements of $\C^n$ by left multiplication.
Clearly this action carries an action of $I^*(P;\C)$ on $\mathbb{P}(\C^n)$, whose orbits are described as follows (see \cite[Theorem 4.2]{BoloSenti}).

Let $\{e_1,\ldots, e_n\}$ be the canonical basis of $\C^n$ and $V_I:=\spn_{\C}\left\{e_i:i\in I\right\}$, for any subset $I\subseteq [n]$. An orbit of the action of $I^*(P;\C)$ on $\mathbb{P}(\C^n)$ is of the form
  $$C_I:=\mathbb{P}(V_I)\setminus \bigcup\limits_{i \in \mathrm{max}(I)}\mathbb{P}\left(V_{I \setminus \{i\}}\right),$$
   for any $I\in \JP\setminus\{\varnothing\}$, and the set of orbits $\{C_I:I\in \JP\setminus\{\varnothing\}\}$ is a
    partition of $\mathbb{P}(\C^n)$.
 The Zariski closure of $C_I$ is a projective space, given by $$\overline{C}_I=\biguplus\limits_{H \in \mathcal{J}(I)\setminus \{\varnothing\}} C_H=\mathbb{P}\left(V_I\right),$$ for all $I\in \mathcal{J}(P)$.

The notion of incidence stratification of a subset of a projective space, introduced in \cite{BoloSenti}, includes some known affine ones, such as the stratifications given by Schubert varieties in Grassmannians and flag varieties (see \cite[Propositions 4.16 and 5.5]{BoloSenti}); the following is the formal definition.
\begin{dfn} \label{def incidence strat}
 Let $X \subseteq \mathbb{P}(\C^n)$ and $P$ be a poset of cardinality $n$.
 The set
 $$\left\{\overline{C}_I\cap X : I \in \mathcal{J}(P) \right\} \setminus \left\{\varnothing \right\}$$ is an \emph{incidence stratification} of $X$.
\end{dfn}
For a projective variety $X$, $A_*X$ denotes the group of $k$-\emph{cycles modulo rational equivalence} on $X$ (see \cite[Chapter 1]{Fulton}). We let $A_*(X;\Q):=A_*X \otimes_{\Z} \Q$ to be the \emph{Chow vector space} of $X$ over $\Q$.

\section{The projective variety $\mathrm{Gr}_f(k,n)$} \label{sezione gr}

Immanant varieties, which we are going to define in the following subsection, are particular cases of a general construction that we describe here.

Let $k,n \in \N$, $V$ be an $n$-dimensional complex vector space and $\{e_1,\ldots,e_n\}$ a basis. Let $\{e_x: x\in [n]^k\}$ be the corresponding basis of $V^{\otimes k}$, where $e_x:=e_{x_1}\otimes \ldots \otimes e_{x_k}$ and $x_i$ is the projection on the $i$-th component of $x$,
 for all $x \in [n]^k$.
 For $f \in \End\left(V^{\otimes k}\right)$, in order to introduce the main objects of our study, define a function
 $$\widehat{f}: \seg(k,n)\setminus \PP(\ker(f)) \rightarrow \PP\left(V^{\otimes k}\right)$$ by setting
 $\widehat{f}([v_1 \otimes \ldots \otimes v_k])= \left[f(v_1 \otimes \ldots \otimes v_k)\right]$,
 for all $v_1,\ldots, v_k \in V \setminus \{0\}$ such that
 $f(v_1 \otimes \ldots \otimes v_k)\neq 0$.

\begin{dfn} \label{definizione}
 Let $f\in \End\left(V^{\otimes k}\right)$.  The algebraic set $\gr_f(k,n) \subseteq \PP(V^{\otimes k})$ is defined  by setting
 $$\gr_f(k,n):=\overline{\imm(\,\widehat{f}\,)},$$
 where the overline stands for the Zariski closure.
\end{dfn}
The notation $\gr_f(k,n)$ is motivated by the fact that, for suitable choices of the function $f$, this construction leads to combinatorial and geometrical notions naturally appearing in the study of Grassmannian varieties.
We observe that $\gr_f(k,n) \subseteq \PP\left(\imm(f)\right)$ and that
the set $\imm(\,\widehat{f}\,)$ is described by a system of $n^k$
parametric polynomial equations with $kn$ parameters. When $\gr_f(k,n) \neq \varnothing$, i.e. $f \neq 0$, by a standard result (see e.g. \cite[Proposition 4.5.5]{cox}), it is irreducible, hence a projective variety.

\begin{prop} \label{teorema irriducibile}
Let $f\in \End\left(V^{\otimes k}\right)\setminus\{0\}$; then the algebraic set $\gr_f(k,n)$ is a projective variety.
\end{prop}

In several cases, the set $\imm(\,\widehat{f}\,)$ is already closed in the Zariski topology, as the next result shows.

\begin{prop} \label{corollario chiuso}
Let $f\in \End\left(V^{\otimes k}\right)$; then
    $$\seg(k,n)\cap \PP(\ker(f))=\varnothing \,\, \Rightarrow \,\, \gr_f(k,n)=\imm(\,\widehat{f}\,).$$
\end{prop}

\begin{proof}
By \cite[Theorem 5.2.2]{Shafarevich}, given an algebraic set $X \subseteq \PP(W)$ and $W$ a complex vector space, if $F : X \rightarrow \PP(W)$ is a polynomial function, then $F$ is closed. Since $\seg(k,n)\cap \PP(\ker(f))=\varnothing$, it holds that $\widehat{f}$ is a polynomial function from $\seg(k,n)$ to $\PP\left(V^{\otimes k}\right)$. From the fact that $\seg(k,n)$ is a projective variety it follows that $\widehat{f}$ is closed.
\end{proof}

The simplest example of these varieties corresponds to $f=\id_{V^{\otimes k}}$; in this case $\imm(\,\widehat{f}\,)=\gr_f(k,n)=\seg(k,n)$.
By using Proposition \ref{corollario chiuso}, it is not difficult to show that if $f \in \End\left(V^{\otimes k}\right)$ is invertible, then $\gr_f(k,n) \simeq \seg(k,n)$ as projective varieties. The general situation is summarized in the following result.

\begin{prop}\label{prop invertibile}
Let $f,g \in \End\left(V^{\otimes k}\right) \setminus \{0\}$ with $g$ invertible. Then $\gr_{g\circ f}(k,n) \simeq \gr_f(k,n)$, as projective varieties.
\end{prop}
\begin{proof} The isomorphism is immediate since the parametric equations
of $\gr_{g\circ f}(k,n)$ in the basis $\{g(e_x):x\in [n]^k\}$
are the ones of $\gr_f(k,n)$ in the basis
$\{e_x: x\in [n]^k\}$.
\end{proof}

\begin{oss}
  For $g\in \End\left(V^{\otimes k}\right)$ invertible, $\gr_{f\circ g}(k,n)$ and  $\gr_f(k,n)$ could be not isomorphic as projective varieties.
  For example, let $k=2$, $n=4$, $\leqslant_{\lex}$ be the lexicographic order on $[4]^2$, $f(e_x)=e_{(x_1,x_2)}-e_{(x_2,x_1)}$,
  and $$g(e_x)= \left\{
                 \begin{array}{ll}
                   \sum\limits_{y\leqslant_{\lex} x}e_y, & \hbox{if $x\leqslant_{\lex} (3,1)$;} \\
                   e_x, & \hbox{otherwise,}
                 \end{array}
               \right. $$
  for all $x\in [4]^2$. Then $\gr_f(2,4)$ is the complex Grassmannian $\gr_{\C}(2,4)$ and $\gr_{f\circ g}(2,4)=\PP\left(\imm(f)\right) \simeq \PP(\C^6)$.
\end{oss}

Now we provide an example of a variety $\gr_f(k,n)$ corresponding to a not invertible endomorphism.

\begin{ex}
  Let $k=3$, $n=2$ and $f\in \End\left(V^{\otimes 3}\right)$
  defined by setting
  $$f(e_x)= \left\{\begin{array}{ll} e_{(2,1,2)}+e_{(2,2,1)}+e_{(2,2,2)}, & \hbox{if $x\in \left\{(2,1,2), (2,2,1), (2,2,2)\right\}$;} \\
e_x, & \hbox{otherwise,}\end{array} \right.$$
for all $x\in [2]^3$. Then $\imm(\,\widehat{f}\,)$ is described by the following parametric equations:
$$\left\{
  \begin{array}{ll}
    x_{111} = a_1b_1c_1 \\
    x_{112} = a_1b_1c_2 \\
    x_{121} = a_1b_2c_1 \\
    x_{122} = a_1b_2c_2 \\
    x_{211} = a_2b_1c_1 \\
    x_{212} = a_2b_1c_2+a_2b_2c_1+a_2b_2c_2\\
    x_{221} = a_2b_1c_2+a_2b_2c_1+a_2b_2c_2\\
    x_{222} = a_2b_1c_2+a_2b_2c_1+a_2b_2c_2
  \end{array}
\right.$$
With the help of Sagemath, the Gr\"obner basis method for implicitization by using the monomial order degrevlex (see for instance \cite{cox}) provides the following Cartesian equations for $\gr_f(3,2)$:
$$\left\{
  \begin{array}{ll}
    (x_{112} + x_{121} + x_{122})x_{211} - x_{111}x_{222} = 0 \\
    x_{112}x_{121} - x_{111}x_{122}=0 \\
    x_{212} - x_{222} = 0 \\
    x_{221} - x_{222} = 0
  \end{array}
\right.$$ This is a projective variety of dimension $3$, obtained as intersection of two hypersurfaces of $\PP\left(\imm(f)\right) \simeq \PP(\C^6)$.
Notice that the last two equations are the equations of $\PP\left(\imm(f)\right)$ in $\PP\left(V^{\otimes 3}\right)$.
The set $\imm(\,\widehat{f}\,)$ is not Zariski-closed. In fact, if $v:=e_{(1,2,2)}+e_{(2,1,1)}-e_{(1,2,2)}$, then $[v] \in \gr_f(3,2) \setminus \imm(\,\widehat{f}\,)$.
\end{ex}

We end this section by showing that some varieties $\gr_f(k,n)$ are actually Cartesian products. For $h\in \N$ and $a \in \N^h$, let $a(i):=\sum_{j=1}^i a_j$, for all $i\in [h]$, and $k:=a(h)$. Given $f_i \in \End\left(V^{\otimes a_i}\right)$, for all $i\in [h]$, we define
$f_1 \otimes \ldots \otimes f_h  \in \End\left(V^{\otimes k}\right)$ by setting
$$(f_1 \otimes \ldots \otimes f_h)(e_x):= f_1(e_{x_{[1,a(1)]}}) \otimes \ldots \otimes f_h(e_{x_{[a(h-1)+1,a(h)]}}),$$
for all $x\in [n]^k$, where, for $i,j\in [k]$, $i\leqslant j$,
$x_{[i,j]}:=(x_i,\ldots,x_j)\in [n]^{j-i+1}$.

 \begin{prop} \label{teorema segri}
Let
$h\in \N$, $a \in \N^h$ and
$k:=a(h)$. Let $f_i \in \End\left(V^{\otimes a_i}\right) \setminus \{0\}$, for all $i\in [h]$. Then, as projective varieties,
$$\gr_{f_1 \otimes \ldots \otimes f_h}(k,n) \simeq \gr_{f_1}(a_1,n) \times \ldots \times \gr_{f_h}(a_h,n).$$
\end{prop}
The following corollary provides a realization of a product of projective spaces as a variety $\gr_f(k,n)$.
\begin{cor} \label{corollario segri}
Let $f_1,\ldots, f_k \in \End(V)\setminus \{0\}$. Then, as projective varieties,
$$\gr_{f_1 \otimes \ldots \otimes f_k}(k,n) \simeq \PP(\C^{b_1}) \times \PP(\C^{b_2}) \times \ldots \times \PP(\C^{b_k}),$$ where $b_i:=\dim\left(\imm(f_i)\right)$,
for all $i\in [k]$.
\end{cor}

\subsection{Immanant varieties}
In this section we introduce the class of varieties $\gr_f(k,n)$ in which we are mostly interested, namely the ones corresponding to an endomorphism of $V^{\otimes k}$ arising from an action of the symmetric group $S_k$.
For $k>0$, a permutation  $w \in S_k$ induces a graded automorphism of the poset $[n]^k$, where the action is defined by
$$w(x)=\left(x_{w^{-1}(1)},\ldots,x_{w^{-1}(k)}\right),$$  for all $x\in [n]^k$. Hence $w\in S_k$ acts on $V^{\otimes k}$ by setting
$w(e_x)=e_{w(x)}$, for all $x\in [n]^k$.
This action
induces an algebra morphism
$\gamma:\C[S_k] \rightarrow \End\left(V^{\otimes k}\right)$, where $\C[S_k]$ is the group algebra over $\C$ of $S_k$.

Let $G \subseteq S_k$ be a group. It is clear that the group algebra $\C[G]$ is a subalgebra of $\C[S_k]$. For $P\in \C[S_k]$,
we set $P^{(n)}:=\gamma(P) \in \End\left(V^{\otimes k}\right)$; hence, if $P=\sum_{g \in G}a_g g$, with $a_g \in \C$ for all $g\in G$, we have $$P^{(n)}(e_x)=\sum_{g \in G}a_g e_{g(x)},$$ for every $x \in [n]^k$. The set $\{P^{(n)}: P\in \C[S_n]\}$ coincides with the set of endomorphisms of $V^{\otimes k}$ commuting with the action of $\gl(V)$ on  $V^{\otimes k}$ defined by setting $h(v_1 \otimes \ldots \otimes v_k)=h(v_1) \otimes \ldots \otimes h(v_k)$, for all $h\in \gl(V)$, $v_1,\ldots, v_k \in V$ (see e.g. \cite[Theorem 8.2.8]{scarabotti}).

For a group $G$, let $\mathbf{1}_G : G \rightarrow \{1\}$ be its trivial character. If $\chi_1 : G \rightarrow \C$ and $\chi_2 : G \rightarrow \C$ are characters, their scalar product is defined by $\langle \chi_1, \chi_2 \rangle := \sum_{g\in G} \chi_1(g)\chi_2(g^{-1})$.
If $\chi$ is a simple character of a subgroup $G\subseteq S_k$, then the element
$$P_{\chi}:=\frac{\chi(e)}{|G|} \sum_{g\in G}\chi(g^{-1})g \in \C[S_k]$$
is an idempotent corresponding to the induced representation $\mathrm{Ind}^{S_k}_{G}(\chi)$ of $S_k$. For such idempotents, we write $\gr_{\chi}(k,n)$
instead of $\gr_{P_\chi^{(n)}}(k,n)$. The following are well-known examples of varieties recovered in this way.

$\,$

\textbf{Segre embeddings:} if $G=\{e\}$, the trivial subgroup of $S_k$, there is only the trivial character $\chi:=\mathbf{1}_G$. In this case $P_{\chi}^{(n)}=\id_{V^{\otimes k}}$ and $\gr_\chi(k,n)=\seg(k,n)$.

$\,$

\textbf{Grassmannians:} if $G=S_k$ and $\chi$ is its alternating character,
then $$P_\chi= \frac{1}{k!}\sum_{\sigma \in S_k}\mathrm{sgn}(\sigma)\sigma,$$ where $\mathrm{sgn}(\sigma)$ denotes the sign of the permutation $\sigma$. Hence, as vector spaces, $\imm\left(P_{\chi}^{(n)}\right) \simeq \bigwedge^{k} V$ and, for $k \leq n$, $\gr_{\chi}(k,n)=\imm\left(\widehat{P}^{(n)}_\chi\right)$ is the complex Grassmannian $\gr_{\C}(k,n)$; for $k>n$ it is clear that $\gr_{\chi}(k,n)=\varnothing$.

$\,$

{\bf Chow varieties} $G(1,k,n)$: another important class of varieties are the so-called {\em Chow varieties} $G(1,k,n)$, the projectivization of the set of homogeneous polynomial of degree $k$ in $n$ variables factorizing in polynomials of degree $1$.
These varieties are recovered as follows.
If $G=S_k$ and $\chi:=\mathbf{1}_{S_k}$,
then $$P_\chi= \frac{1}{k!}\sum_{\sigma \in S_k}\sigma.$$ Hence, as vector spaces, $\imm\left(P_{\chi}^{(n)}\right) \simeq \mathrm{Sym}^{k} V$ and $\gr_{\chi}(k,n)=\imm\left(\widehat{P}^{(n)}_\chi\right)$ is the Chow variety $G(1,k,n)$. For more details see e.g. \cite[Chapter~4]{discriminant} and \cite[Section~8.6]{Landsberg}.

$\,$

The following construction realizes a Cartesian product of projective spaces differently with respect to Corollary \ref{corollario segri}.

$\,$

{\bf Cartesian product of projective spaces}: let $h\in \N$,
$a \in \N^h$ and $k:=a(h)$, where $a(j):=\sum_{i=1}^ja_i$, for all $j\in [h]$. The symmetric group $S_k$ is generated by the simple transpositions
$\{s_1,\ldots,s_{k-1}\}$; for $i,j \in [k-1]$, $i\leqslant j$, define the parabolic subgroup $$S_{[i,j]}:= \left\{
  \begin{array}{ll}
    \langle s_i,s_{i+1},\ldots, s_{j-1}\rangle, & \hbox{ if $i<j$;} \\
    \{e\}, & \hbox{ if $i=j$.}
\end{array}\right.$$ We define the parabolic subgroup $G_a \subseteq S_k$ by
$$G_a:=S_{[1,a(1)]} \times S_{[a(1)+1,a(2)]} \times \ldots \times S_{[a(h-1)+1,a(h)]}.$$ Hence
$P_{\textbf{1}_{G_a}}=\frac{1}{a_1!\cdots a_h!}\sum_{g\in G_a}g \in \C[S_k]$ is the idempotent corresponding to the Young module $M^a$ (see e.g. \cite[Section 3.6.2]{scarabotti}).
Since, for $k\geqslant 1$, the Chow variety $G(1,k,2)$ is isomorphic to $\PP(\C^{k+1})$ (because any homogeneous polynomial of positive degree in two variables factorizes as product of degree one polynomials), we obtain
\begin{equation} \label{prodotto proiettivi}
    \gr_{\textbf{1}_{G_a}}(k,2) \simeq \PP(\C^{a_1+1}) \times \PP(\C^{a_2+1}) \times \ldots \times \PP(\C^{a_h+1}).
\end{equation}

 For a simple character $\chi$ of $G\subseteq S_k$, define $V^{\otimes k}_{\chi}:=\imm\left({P^{(n)}_{\chi}}\right)$.
 We denote by $\mathrm{Res}^G_{H} (\chi)$
 the restriction to a group $H\subseteq G$ of a character $\chi$ of $G$. Given an action of a group $G$ on a set $X$, let
 $G_x$ be the \emph{isotropy group} of the element $x\in X$,
 i.e. $G_x:=\{g\in G: g(x)=x\}$.

 A set of generators and the dimension of the vector space $V^{\otimes k}_{\chi}$ are known, see \cite[Eq. 6.13 and 6.23]{Merris}.
\begin{thm} \label{teorema span}
Let $G\subseteq S_k$ be a subgroup and $\chi$ a simple character of $G$. Then
$P^{(n)}_{\chi}(e_x)=0$ if and only if $\left \langle\mathrm{Res}^G_{G_x} (\chi) , \mathbf{1}_{G_x} \right\rangle = 0$, for every $x \in [n]^k$,
and
$$\dim\left(V^{\otimes k}_{\chi}\right)=\frac{\chi(e)}{|G|}\sum \limits_{g \in G}\chi(g)n^{c(g)},$$ where $c(g)$ is the number of cycles of the permutation $g \in S_k$.
\end{thm}
In particular, by Theorem \ref{teorema span}, we have that
$$V^{\otimes k}_{\chi}=\spn_{\C}\left\{ P^{(n)}_{\chi}(e_x) : \left\langle \mathrm{Res}^G_{G_x} (\chi) , \mathbf{1}_{G_x} \right\rangle \neq 0 \right\}.$$
For a more explicit description of the spanning set in Theorem \ref{teorema span} when $G=S_k$, see \cite[Theorem 6.37]{Merris}.
In general, this set is not a basis, as shown in the following example. Moreover, this is an example of immanant variety where $\gr_{\chi}(k,n) \neq \imm\left(\widehat{P}_{\chi}^{(n)}\right)$.

\begin{ex}
Let $n=2$, $k=3$, $G=S_3$ and $\chi$ the
two-dimensional simple character of $S_3$ corresponding to the partition $(2,1)$. Then
$P_\chi=\frac{1}{3}\left(2e-st-ts\right)$, where $s=213$ and $t=132$ are the simple transpositions (we write permutations in one-line notation). By Theorem \ref{teorema span}, $\dim\left(V^{\otimes k}_{\chi}\right)=4$. We have that $P^{(2)}_{\chi}(e_{(1,1,1)})=P^{(2)}_{\chi}(e_{(2,2,2)})=0$ and
\begin{itemize}
    \item $P^{(2)}_{\chi}(e_{(1,1,2)})=\frac{1}{3}\left(2e_{(1,1,2)}-e_{(1,2,1)}-e_{(2,1,1)}\right)$,
    \item $P^{(2)}_{\chi}(e_{(1,2,1)})=\frac{1}{3}\left(2e_{(1,2,1)}-e_{(2,1,1)}-e_{(1,1,2)}\right)$,
    \item $P^{(2)}_{\chi}(e_{(2,1,1)})=\frac{1}{3}\left(2e_{(2,1,1)}-e_{(1,1,2)}-e_{(1,2,1)}\right)$.
\end{itemize} Hence  $P^{(2)}_{\chi}(e_{(1,2,1)})=-P^{(2)}_{\chi}(e_{(1,1,2)})-P^{(3,2)}_{\chi}(e_{(2,1,1)})$ (similarly for $P^{(2)}_{\chi}\left(e_{(2,1,2)}\right)$) and a basis for $V^{\otimes 3}_{\chi}$ is $$\left\{P^{(2)}_{\chi}(e_{(1,1,2)}), P^{(2)}_{\chi}(e_{(2,1,1)}), P^{(2)}_{\chi}(e_{(2,2,1)}), P^{(2)}_{\chi}(e_{(1,2,2)})\right\}.$$
Parametric equations for $\imm\left(\widehat{P}^{(n)}_\chi\right)$ are
$$\left\{
  \begin{array}{ll}
    x_{111} = 0 \\
    x_{112} = 2a_1b_1c_2-a_2b_1c_1-a_1b_2c_1 \\
    x_{121} = 2a_1b_2c_1-a_1b_1c_2-a_2b_1c_1 \\
    x_{122} = 2a_1b_2c_2-a_2b_1c_2-a_2b_2c_1 \\
    x_{211} = 2a_2b_1c_1-a_1b_2c_1-a_1b_1c_2 \\
    x_{212} = 2a_2b_1c_2-a_2b_2c_1-a_1b_2c_2\\
    x_{221} = 2a_2b_2c_1-a_1b_2c_2-a_2b_1c_2\\
    x_{222} = 0
  \end{array}
\right.$$ where $(a_1,a_2),(b_1,b_2), (c_1,c_2)\in \C^2\setminus \{(0,0)\}$. Using Gr\"{o}bner basis method for implicitization, we find the following Cartesian equations for $\gr_{\chi}(3,2)$:

$$\left\{
  \begin{array}{ll}
    x_{112} + x_{121} + x_{211} = 0 \\
    x_{221} + x_{212} + x_{122} = 0 \\
    x_{111} = 0 \\
    x_{222} = 0
  \end{array}
\right.$$

These equations are the ones for $\PP(V^{\otimes 3}_\chi)$ in
$\PP(V^{\otimes 3})$; hence $\gr_{\chi}(3,2)=\PP(V^{\otimes 3}_\chi) \simeq \PP(\C^4)$.
It is not difficult to see that $\imm\left(\widehat{P}^{(n)}_\chi\right)\subsetneq \gr_{\chi}(3,2)$;
in fact it can be checked by hand that $$\left[P^{(2)}_{\chi}(e_{(1,1,2)})-P^{(2)}_{\chi}(e_{(2,1,1)})\right] \in \PP(V^{\otimes 3}_\chi) \setminus  \imm\left(\widehat{P}^{(n)}_\chi\right).$$
\end{ex}

$\,$

In the following, we are going to introduce a generalization of the immanant of a square matrix to matrices with arbitrary size, depending on a subgroup $G \subseteq S_k$ and a simple character of $G$. Our definition extends to submatrices the notion of generalized matrix function (see \cite[Chapter 7]{Merris}).

Let $k,m,n \in \N$; recall that $V$ is an $n$-dimensional complex vector space. If $W$ is an $m$-dimensional complex vector space and $f \in \Hom(V,W)$, the element $f^{\otimes k} \in \Hom(V^{\otimes k},W^{\otimes k})$ is defined by setting $$f^{\otimes k}(e_x)=f(e_{x_1}) \otimes \ldots \otimes f(e_{x_k}),$$ for all $x\in [n]^k$.

$\,$

\begin{oss}
   When $W=V$, we have that $(f\circ g)^{\otimes k}=f^{\otimes k}\circ g^{\otimes k}$, for all $f,g\in \End(V)$, and $(\id_V)^{\otimes k}=\id_{V^{\otimes k}}$, i.e. the function $f \mapsto f^{\otimes k}$ defines a monoid morphism $\End(V) \rightarrow \End(V^{\otimes k})$.
\end{oss}

 Given a matrix $M\in \mat_{m,n}(\C)$,  the matrix
 $M^{\otimes k} \in \mat_{m^k,n^k}(\C)$ is defined by setting
 $$\left(M^{\otimes k}\right)_{x,y}= \prod_{i=1}^k M_{x_i,y_i},$$ for all $x \in [m]^k$, $y \in [n]^k$, i.e. if $M$ is the matrix associated to $f \in \Hom(V,W)$ with respect to bases of $V$ and $W$, the matrix associated to $f^{\otimes k}$ is $M^{\otimes k}$.
 In the following definition we extend the notion of immanant of a square matrix.

 \begin{dfn} \label{def immanant}
   Let $x\in [m]^k$, $y\in [n]^k$ and $\chi$ a simple character of a group $G\subseteq S_k$. The $\chi_{x,y}$-\emph{immanant} of a matrix $M \in \mat_{m,n}(\C)$ is defined by
   $$\chi_{x,y}(M):=  \sum \limits_{g\in G}\chi(g)\left(M^{\otimes k}\right)_{g(x),y}.$$
 \end{dfn} Let $M \in \mat_{n,n}(\C)$, $k=n$ and $x=y=(1,\ldots,n)=:\vec{n}$. In this setting, for $G=S_n$, we recover some well-known numbers associated to $M$.
\begin{itemize}
\item If $\chi$ is the alternating character of $S_n$, then $\chi_{\vec{n},\vec{n}}(M)$ is the \emph{determinant} of the matrix $M$.
\item If $\chi:=\mathbf{1}_{S_n}$, then $\chi_{\vec{n},\vec{n}}(M)$ is the \emph{permanent}  of the matrix $M$.
\item If $\chi$ is any simple character of $S_n$, then $\chi_{\vec{n},\vec{n}}(M)$ is the so-called \emph{immanant} of the matrix $M$, see for instance \cite{stembridge} and references therein.
\end{itemize}

\begin{ex} \label{proponi allo slumato}
Let $M \in \mat_{2,3}(\C)$ be the generic matrix
$$M=\begin{pmatrix}
a_{11} & a_{12} & a_{13}\\
a_{21} & a_{22} & a_{23}
\end{pmatrix}.$$ Then, by ordering $[2]^2$ and $[3]^2$ lexicographically, the matrix $M^{\otimes 2}$ is $$ \begin{pmatrix}
a_{11}^2 & a_{11}a_{12} & a_{11}a_{13} & a_{11}a_{12} & a_{12}^2 & a_{12}a_{13} & a_{11}a_{13} & a_{12}a_{13} & a_{13}^2\\
a_{11}a_{21} & a_{11}a_{22} & a_{11}a_{23} & a_{12}a_{21} & a_{12}a_{22} & a_{12}a_{23} & a_{13}a_{21} & a_{13}a_{22} & a_{13}a_{23}\\
a_{21}a_{11} & a_{21}a_{12} & a_{21}a_{13} & a_{22}a_{11} & a_{22}a_{12} & a_{22}a_{13} & a_{23}a_{11} & a_{23}a_{12} & a_{23}a_{13}\\
a_{21}^2 & a_{21}a_{22} & a_{21}a_{23} & a_{22}a_{21} & a_{22}^2 & a_{22}a_{23} & a_{23}a_{21} & a_{23}a_{22} & a_{23}^2
\end{pmatrix}.$$

 If $\chi$ is the trivial character of $S_2$, we have $\chi_{(2,2),(2,3)}(M)=2a_{22}a_{23}$. If $\chi$ is the alternating character of $S_2$, we have $\chi_{(1,2),(1,3)}(M)=a_{11}a_{23}-a_{21}a_{13}$ and $\chi_{(2,2),(2,3)}(M)=0$.
\end{ex}


 Let $P\in \C[S_k]$; it is clear that $$f^{\otimes k} \circ P^{(n)} = P^{(m)} \circ f^{\otimes k},$$  i.e. $f^{\otimes k}\in \Hom_{\C[S_k]}\left(V^{\otimes k},W^{\otimes k}\right)$, for all $f\in \Hom(V,W)$.

 The $\chi_{x,y}$-immanant of a matrix $M$ is related to the $(x,y)$-entry of the matrix $M^{\otimes k}P^{(n)}_\chi$, as the next proposition asserts.
\begin{prop} \label{proposizione immananti} Let $x\in [m]^k$, $y\in [n]^k$, $\chi$ a simple character of a group $G\subseteq S_k$ and $M\in \mat_{m,n}(\C)$. Then
   $$\left(M^{\otimes k}P^{(n)}_\chi\right)_{x,y}=\frac{\chi(e)}{|G|}\chi_{x,y}(M).$$
 \end{prop}
\begin{proof}
  Let $y\in [n]^k$ and $\{w_i: i \in [m]\}$ be a basis of the $m$-dimensional vector space $W$. We have that
 \begin{eqnarray*}
   M^{\otimes k}P^{(n)}_\chi e_y &=& P^{(m)}_\chi M^{\otimes k}e_y \\
   &=&   \frac{\chi(e)}{|G|} \sum \limits_{g\in G}\chi(g^{-1})\sum \limits_{x\in [m]^k}  \left(M^{\otimes k}\right)_{x,y}w_{g(x)} \\
   &=&  \frac{\chi(e)}{|G|} \sum \limits_{g\in G}\chi(g^{-1})\sum \limits_{x\in [m]^k} \left(M^{\otimes k}\right)_{g^{-1}(x),y}w_x \\
   &=&  \frac{\chi(e)}{|G|} \sum \limits_{x\in [m]^k}\left( \sum \limits_{g\in G}\chi(g) \left(M^{\otimes k}\right)_{g(x),y}\right)w_x.
 \end{eqnarray*}
\end{proof}

 \begin{oss}
 For $P\in \C[S_k]$, we can define, by restriction, an element $f^{\otimes k}_P \in  \Hom\left(\imm(P^{(n)}),\imm(P^{(m)})\right)$; if $P=P_{\chi}$ for some simple character $\chi$ of a group $G\subseteq S_k$,
 then $f^{\otimes k}_P \in \Hom_{\C[G]}\left(\imm(P^{(n)}),\imm(P^{(m)})\right)$, since $gP_{\chi}=P_{\chi}g$, for all $g\in G$, and then $\imm(P^{(n)})$, $\imm(P^{(m)})$ are $\C[G]$-modules. See \cite{induced operators} for a more extended treatment on such induced morphisms.
 \end{oss}
Now we provide parametric equations for the set $\imm\left(\widehat{P}^{(n)}_\chi\right)$ in terms of immanants.
Let $A\in \mat_{n,k} (\C)$ defined by
$$A:=\left(
          \begin{array}{cccc}
            a_{11} & a_{12} & \cdots & a_{1k} \\
            a_{21} & a_{22} & \cdots & a_{2k} \\
            \vdots & \vdots & \cdots & \vdots \\
            a_{n1} & a_{n2} & \cdots & a_{nk} \\
          \end{array}
        \right)
$$ where none of the columns is the zero vector.
\begin{thm} \label{equazioni} Let $k,n \in \N$ and $\chi$ be a simple character of
a group $G \subseteq S_k$. Then
  the set $\imm(\widehat{P}^{(n)}_\chi) \subseteq \PP\left(V^{\otimes k}\right)$ is described by the parametric equations
  $$\left\{x_z = {\chi}_{z,(1,\ldots,k)}(A) : z \in [n]^k\right\},$$ where $A$ is the generic matrix defined above.
\end{thm}
\begin{proof} Let $v_i=a_{1i}e_1+\ldots +a_{ni}e_n$, for all $i\in [k]$, and let $W$ be a vector space with basis $\{\tilde{e}_1,\ldots,\tilde{e}_k\}$. Then
$A \in \Hom(W,V)$ and, by Proposition \ref{proposizione immananti},
\begin{eqnarray*}
  P^{(n)}_\chi(v_1 \otimes \ldots \otimes v_k) &=& P^{(n)}_\chi\left( (A\tilde{e}_1)\otimes \ldots \otimes (A\tilde{e}_k)\right) \\
  &=& P^{(n)}_\chi A^{\otimes k} \tilde{e}_{(1,...,k)} \\
  &=&  A^{\otimes k}P^{(k)}_\chi \tilde{e}_{(1,...,k)} \\
  &=&\frac{\chi(e)}{|G|} \sum \limits_{x\in[n]^k}\chi_{x,(1,\ldots,k)}(A)\tilde{e}_x.
\end{eqnarray*}
\end{proof}

The statement of Theorem \ref{equazioni} leads us to the following definition.
\begin{dfn} \label{definizione varieta immananti}
If $\chi$ is a simple character of
a group $G \subseteq S_k$, we call $\gr_{\chi}(k,n)$
an \emph{immanant variety}.
\end{dfn}
We observe that, if two subgroups $G\subseteq S_k$ and $H\subseteq S_k$
are conjugated in $S_k$, i.e. $H=\sigma G\sigma^{-1}$ for some $\sigma \in S_k$, and $\chi$ is a character of $G$,
then, as projective varieties, $\gr_{\chi}(k,n)\simeq \gr_{\chi^\sigma}(k,n)$, where $\chi^\sigma(h):=\chi(\sigma^{-1} h \sigma)$, for all $h\in H$. In fact $P_{\chi^{\sigma}}=\sigma P_{\chi}\sigma^{-1}$. Let $f:=\gamma(P_{\chi}\sigma^{-1}) \in \End\left(V^{\otimes k}\right)$, where $\gamma:\C[S_k] \rightarrow \End\left(V^{\otimes k}\right)$ is the algebra morphism introduced above. Then
$\imm(\,\widehat{f}\,)=\imm(\widehat{P}^{(n)}_{\chi})$ and
$$\gr_{\chi^{\sigma}}(k,n)=\gr_{\gamma(\sigma)\circ f}(k,n)\simeq \gr_f(k,n)=\gr_{\chi}(k,n),$$ by Proposition \ref{prop invertibile}. In the following example we see that $G \simeq H$ and $\chi_G \simeq \chi_H$ do not imply $\gr_{\chi_G}(k,n) \simeq \gr_{\chi_H}(k,n)$ as projective varieties.

\begin{ex} \label{MA CHE E' STA PECIONATA}
  Let $G:=\{e,2134\} \subseteq S_4$ and $H:=\{e,4321\} \subseteq S_4$.
  Then $\gr_{\mathbf{1}_G}(4,2) \simeq \PP(\C^3) \times \PP(\C^2) \times \PP(\C^2)$, by \eqref{prodotto proiettivi}, because $G=G_{(2,1,1)}$. Hence it is a smooth four-dimensional projective variety. On the other hand, we checked by using Macaulay2 \cite{M2} that $\gr_{\mathbf{1}_H}(4,2)$ is a singular four-dimensional projective varieties.
  Then, as projective varieties, they are not isomorphic.
\end{ex}

\section{One-dimensional characters and $\chi$-matroids} \label{sezione unidimensionale}
In this section we restrict our attention to one-dimensional characters
of a subgroup $G \subseteq S_k$. This includes, for example, the trivial character of $G$ and all the simple characters of an abelian group.

For $x\in [n]^k$ we let $O_x:=\{g(x):g \in G\}\subseteq [n]^k$ and $\preccurlyeq_{\lex}$
the lexicographic order on $O_x$. Then we define
\begin{equation} \label{barra}
    \overline{x}:=\min\left(O_x,\preccurlyeq_{\lex}\right).
\end{equation}
For completeness, we give a proof of the following theorem, ensuring that, for one-dimensional characters, the spanning set of Theorem \ref{teorema span} provides a basis in a canonical way. It is known in another formulation, see \cite[Corollary 6.32]{Merris}.
\begin{thm} \label{teorema base}
Let $G\subseteq S_k$ be a subgroup and $\chi$ a one-dimensional character of $G$, i.e. a group morphism $\chi : G \rightarrow \C\setminus \{0\}$. Then $$\left\{P^{(n)}_{\chi}(e_{\overline{x}}) : x \in [n]^k,\, G_x \subseteq \ker(\chi) \right\}$$
is a basis of $V^{\otimes k}_{\chi}$.
\end{thm}
\begin{proof}
By Theorem \ref{teorema span} we have that $P^{(n)}_{\chi}(e_x)=0$
if and only if $\sum\limits_{g \in G_x} \chi(g)=0$ and this is equivalent to
$G_x \not \subseteq \ker(\chi)$. In fact, it is straightforward to see that $\sum\limits_{g \in G_x} \chi(g)=0$ implies $G_x \not \subseteq \ker(\chi)$. On the other hand,
let $h\in G_x \setminus \ker(\chi)$; therefore $\sum\limits_{g \in G_x} \chi(g)=\sum\limits_{g \in G_x} \chi(hg)=\chi(h)\sum\limits_{g \in G_x} \chi(g)$ and we find that
$(1-\chi(h))\sum\limits_{g \in G_x} \chi(g)=0$. This implies
$\sum\limits_{g \in G_x} \chi(g)=0$.

Let $x,y \in [n]^k$ with $G_x \subseteq \ker(\chi)$ and $G_y \subseteq \ker(\chi)$. Moreover assume $y=g(x)$ for some $g\in G$. Notice that
$$P^{(n)}_{\chi}(e_x)=\frac{1}{|G/G_x|}\sum \limits_{c \in G/G_x} \chi(\tilde{c}^{-1})e_{\tilde{c}(x)},$$ for all $x\in [n]^k$, where $\tilde{c}$ is any representative of the class $c \in G/G_x$. By using this formula, it can be easily shown that
$P^{(n)}_{\chi}(e_y)=\chi(g)P^{(n)}_{\chi}(e_x)$. It follows that it is sufficient to choose an element for each orbit $O_x$, for instance $\overline{x}$.
Since $\overline{x}_1,\overline{x}_2,...,\overline{x}_h \in [n]^k$ are in $h$ distinct orbits then clearly
$P^{(n)}_{\chi}(e_{\overline{x}_1})$, $P^{(n)}_{\chi}(e_{\overline{x}_2})$, $\ldots$, $P^{(n)}_{\chi}(e_{\overline{x}_h})$ are linearly independent. This concludes the proof.
\end{proof}

We define $$B_{\chi}(k,n):=\left\{\overline{x}: x \in [n]^k, \, G_x \subseteq \ker{\chi}\right\},$$
ordered by setting $$ \mbox{$x \preccurlyeq y$ if and only if $x \leqslant g(y)$},$$ for some $g \in G$, for all $x,y \in B_{\chi}(k,n)$, where $\leqslant$ is componentwise.
We need to prove that this is really a partial order.
\begin{prop}
  The relation $\preccurlyeq$ on $B_{\chi}(k,n)$ is a partial order.
\end{prop}
\begin{proof} We claim that $x \leqslant g(x)$ if and only if $g\in G_x$. One implication is obvious. Let  $x \leqslant g(x)$ for some $g\in G$.
  Let $\rho$ be the rank function of the poset $([n]^k,\leqslant)$. Then $\rho(x)=\rho(g(x))$. This implies that $g(x)=x$.
  \begin{enumerate}
  \item reflexivity: clearly $x \leqslant e(x)=x$ and then  $x \preccurlyeq x$, for all $x\in B_{\chi}(k,n)$.
  \item symmetry: let $x \preccurlyeq y$ and $y\preccurlyeq x$. Then $x \leqslant g(y)$ and $y \leqslant h(x)$, for some $g,h \in G$.
  Therefore $x \leqslant g(y) \leqslant gh(x)$. Hence $x \leqslant gh(x)$ and by our claim we conclude that $gh(x)=x$. Then $x=g(y)$, i.e. $y\in O_x$ and $x=y$.
  \item transitivity: let $x \preccurlyeq y$ and $y \preccurlyeq z$; then $x \leqslant g(y)$ and $y \leqslant h(z)$ for some $g,h\in G$.
  This implies $x \leqslant g(y) \leqslant gh(z)$, i.e. $x \preccurlyeq z$.
\end{enumerate}
\end{proof}
For example, if $G=\{e\}$ is the trivial group,
then $B_{\chi}(k,n)=\left( [n]^k, \leqslant \right)$.
When $\chi$ is the alternating character of $S_k$, the poset $B_{\chi}(k,n)$
is isomorphic to the poset $S_n^{(k)}$ of Grassmannian permutations  with the Bruhat order (see \cite[Proposition 4.9]{BoloSenti}).

\begin{ex}\label{lyndon} [Lyndon words]
  Let $C_k:=\{e, k12 \ldots k-1, \ldots, 23\ldots 1\}\subseteq S_k$ be a cyclic group of order $k$, $\sigma:=k12 \ldots k-1 \in C_k$ and
  $\chi : C_k \rightarrow \C \setminus \{0\}$ the character defined by
  $\chi(\sigma^h)=\exp\left(\frac{2\pi h i}{k}\right)$, for all $h\in [k]$.
  Since $\ker(\chi)=\{e\}$, we have that $B_{\chi}(k,n)$ is the poset of \emph{Lyndon words} of length $k$ over the alphabet $[n]$. The poset $B_{\chi}(k,n)$ has maximum $(n-1,n,\ldots,n)$
  and minimum $(1,\ldots,1,2)$. The cardinality of $B_{\chi}(k,n)$ is given by Witt's formula:

  $$|B_{\chi}(k,n)|=\frac{1}{k}\sum\limits_{d \mid k} \mu(d)n^{k/d},$$ where $\mu$ is the M\"{o}bius function. For more information on Lyndon words, see e.g. \cite[Exercise 7.89]{Stenley}.
  The following is the Hasse diagram of $B_\chi(6,2)$.
    
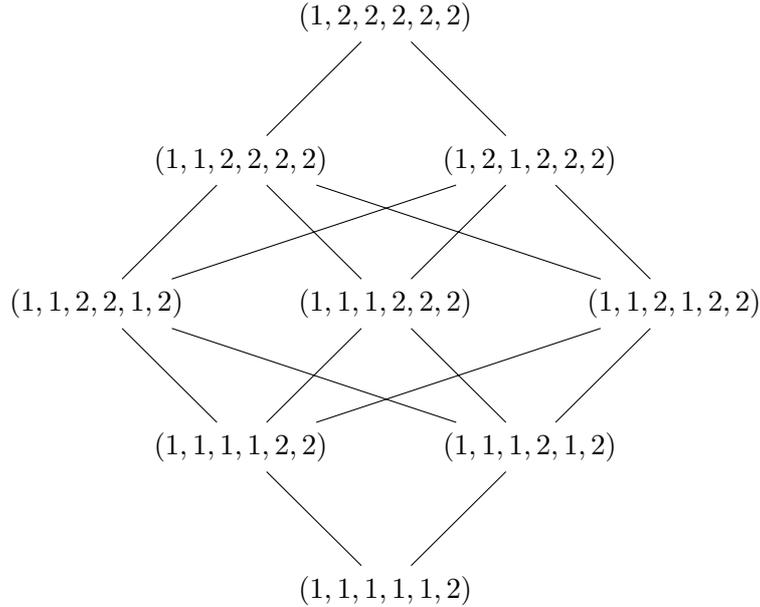
\begin{center}
 \begin{tikzpicture}[scale=1.9]
  \node (a1) at (3,5) {$(1,2,2,2,2,2)$};
  \node (b1) at (2,4) {$(1,1,2,2,2,2)$};
  \node (b2) at (4,4) {$(1,2,1,2,2,2)$};
  \node (c1) at (1,3) {$(1,1,2,2,1,2)$};
  \node (c2) at (3,3) {$(1,1,1,2,2,2)$};
  \node (c3) at (5,3) {$(1,1,2,1,2,2)$};
  \node (d1) at (2,2) {$(1,1,1,1,2,2)$};
  \node (d2) at (4,2) {$(1,1,1,2,1,2)$};
  \node (e1) at (3,1) {$(1,1,1,1,1,2)$};
  \draw (a1) -- (b1) (a1) -- (b2) (b1) -- (c1)
  (b1) -- (c2) (b1) -- (c3) (b2) -- (c1) (b2) -- (c2)
  (b2) -- (c3) (c1) -- (d1) (c1) -- (d2) (c2) -- (d1) (c2) -- (d2) (c3) -- (d1) (c3) -- (d2) (d1) -- (e1) (d2) -- (e1);
\end{tikzpicture}
  \end{center}
\end{ex}

Now we extend the notion of rank $k$ matroids to all one-dimensional characters $\chi$ of a group $G \subseteq S_k$. Let $\sigma \in S_n$; then $\sigma$ acts on $[n]^k$
by letting $$\sigma^*(x):=(\sigma(x_1),\ldots,\sigma(x_k)),$$ for all $x \in [n]^k$. Hence we have an action of $S_n$
on $V^{\otimes k}$, commuting with the previously defined action of $S_k$ on $V^{\otimes k}$.

\begin{dfn} \label{chi matroidi}
Let $X \subseteq B_{\chi}(k,n)$. We say that $X$ is a {\em $\chi$-matroid} if, for every $\sigma \in S_n$, the induced subposet $\left\{\overline{\sigma^*(x)}:x \in X\right\}\subseteq B_{\chi}(k,n)$ has a unique maximum, where for $x\in [n]^k$,  $\overline{x}$ is defined in \eqref{barra}.
\end{dfn}

\begin{oss} \label{recupero matroidi}
A characterization of matroids, due to Gale, is by their maximality property (see \cite[Theorem 1.3.1]{coxeter matroids}). Therefore, if $\chi$ is the alternating character of $G=S_k$, we recover the set of bases of a rank $k$ matroid on the ground set $[n]$.
\end{oss}

For a one-dimensional character $\chi$ of $G \subseteq S_k$ and $[v]\in \PP(V^{\otimes k}_{\chi})$, we define a set $\supp_\chi[v] \subseteq B_{\chi}(k,n)$ by letting $$\supp_\chi[v]:=\left\{x \in B_{\chi}(k,n): a_x \neq 0\right\},$$ where $v=\sum\limits_{x \in B_{\chi}(k,n)}a_x P^{(n)}_{\chi}(e_x)$.

$\,$

We now introduce the \emph{maximality property} for  $\imm\left(\widehat{P}^{(n)}_\chi \right)$, whenever $\chi$ is a one-dimensional character of a subgroup $G\subseteq S_k$.

\begin{dfn} Let $\chi$ be a one-dimensional character of a subgroup $G\subseteq S_k$. We say that $\imm\left(\widehat{P}^{(n)}_\chi \right)$ has the \emph{maximality property} if the induced subposet $\supp_\chi[v]$ is a $\chi$-matroid, for all $[v] \in \imm\left(\widehat{P}^{(n)}_\chi \right)$.
\end{dfn}

In the following example we see a case in which the maximality property does not hold.
\begin{ex} \label{non e chi matroide}
  Let $n=3$, $k=4$ and $G=\{e,g\}=\{e,3412\} \subseteq S_4$ and $\chi: G \rightarrow \C$ defined by setting $\chi(3412)=-1$.
  Then the action of the group $G$ on $V^{\otimes 4}$ is given by
   $g\left(e_{(i_1,i_2,i_3,i_4)}\right)=e_{(i_3,i_4,i_1,i_2)}$, for all $i_1, i_2, i_3, i_4 \in [3]$.
   We have that $P^{(3)}_{\chi}(e_{(3,3,3,3)})=0$ and that the poset $B_{\chi}(4,3)$ has maxima $(2,3,3,3)$ and $(3,2,3,3)$. Let $v:=P^{(3)}_{\chi}\left((e_2+e_3)\otimes (e_2+e_3) \otimes e_3 \otimes e_3\right)$; therefore
   \begin{eqnarray*}
     v &=& P^{(3)}_{\chi}(e_{(2,2,3,3)})+  P^{(3)}_{\chi}(e_{(2,3,3,3)})  +  P^{(3)}_{\chi}(e_{(3,2,3,3)}).
   \end{eqnarray*} Hence $\supp_\chi[v]=\{(2,2,3,3),(2,3,3,3),(3,2,3,3)\}$
   is not a $\chi$-matroid because this poset has two maximal elements.
\end{ex}

Several important varieties have the maximality property.
For example, the Grassmannian $\gr_{\C}(k,n)$ has this property. In fact, the elements of $\gr_{\C}(k,n)$ correspond to representable matroids of rank $k$ and matroids can be defined by their maximality property (\cite[Theorem 1.3.1]{coxeter matroids}). This is also the case for Segre varieties.

\begin{prop} \label{massimalità segre}
The set $\seg(k,n)$ has the maximality property.
\end{prop}
\begin{proof} For $k\in \N$, let $e(k)$ be the identity in $S_k$.
The statement is equivalent to say that $\supp_{\mathbf{1}_{\{e(k)\}}}[v_1 \otimes \ldots \otimes v_k]$ has a unique maximum for all $v_1,\ldots,v_k \in V\setminus \{0\}$. Let $v:=v_1 \otimes \ldots \otimes v_k$; hence
$$\max (\supp_{\mathbf{1}_{\{e(k)\}}}[v])=\left(\max (\supp_{\mathbf{1}_{\{e(1)\}}}[v_1]),\ldots,\max (\supp_{\mathbf{1}_{\{e(1)\}}}[v_k])\right)$$ and this concludes the proof.
\end{proof}
More in general, when $\chi$ is a trivial character, the set $\imm\left(\widehat{P}^{(n)}_\chi \right)$ has the maximality property, as we prove in the next section (see Corollary \ref{corollario massimalita}).
We end this section with one more definition.
\begin{dfn}
We say that a subset $X \subseteq B_{\chi}(k,n)$ is \emph{representable} over $\C$ if $X \in \left\{\supp_\chi[v]: [v]\in \imm\left(\widehat{P}^{(n)}_\chi \right)\right\}$.
\end{dfn}
Notice that $\imm\left(\widehat{P}^{(n)}_\chi \right)$ has the maximality property if and only all representable subsets $X \subseteq B_{\chi}(k,n)$ are $\chi$-matroids.
Differently to the Grassmannian case, there exist representable subsets in $B_{\chi}(k,n)$ which are not $\chi$-matroids, see Example \ref{non e chi matroide}.

\begin{oss} Let $m:=\dim\left(V_\chi^{\otimes k}\right)$.
If $\gr_{\chi}(k,n)=\imm\left(\widehat{P}^{(n)}_\chi \right)$ and the variety $\gr_{\chi}(k,n)$ has the maximality property, then the action on $\PP\left(V_\chi^{\otimes k}\right)$ of the group of invertible diagonal matrices of size $m$  (i.e. the incidence group of the trivial poset of cardinality $m$), provides an incidence stratification of $\gr_{\chi}(k,n)$ whose strata are in bijection with representable $\chi$-matroids. If $\chi$ is the alternating character of $S_k$, we recover the {\em matroidal strata} introduced in \cite{Gelfand}. These strata provide a geometric interpretation of the Tutte polynomials via the $K$-theory of Grassmannians, see \cite{fink}.
\end{oss}

\section{The combinatorics of $\gr_{\mathbf{1}_G}(k,n)$} \label{sezione triviale}

In this section we consider the trivial character $\mathbf{1}_{G}$ of a subgroup $G\subseteq S_k$.
In this case $B_{\mathbf{1}_{G}}(k,n)=\{\overline{x}: x\in [n]^k\}$ because the condition $G_x \subseteq \ker(\mathbf{1}_{G})=G$ (see Theorem \ref{teorema base}) is trivially satisfied. Moreover we have that $\PP\left(\ker\left(P^{(n)}_{\mathbf{1}_G}\right)\right)\cap \seg(k,n)=\varnothing$; in fact we can state the following more general result.

\begin{prop}
Let $P\in \C[S_k]$ be such that $P^{(n)}(e_x) \neq 0$ for all $x \in [n]^k$. Then
$\gr_{P^{(n)}}(k,n)=\imm(\widehat{P}^{(n)})$.
\end{prop}
\begin{proof}
Let $v:=v_1\otimes \ldots \otimes v_k \in V^{\otimes k}$.
Hence $\max \left(\supp_{\mathbf{1}_G}[v]\right)=\{x\} \subseteq [n]^k$, by the maximality property of $\seg(k,n)$. Let $\rho$ be the rank function of $[n]^k$; then
$\rho(y)<\rho(x)=\rho(z)$, for all $y\in \supp_{\mathbf{1}_G}[v] \setminus \{x\}$ and for all $z\in \supp_{\mathbf{1}_G}\left[P^{(n)}(e_x)\right]$, because the action of $S_k$ on $[n]^k$ preserves its rank. This implies $P^{(n)}(v)\neq 0$; hence $\seg(k,n)\cap \PP\left( \ker\left(P^{(n)}\right) \right)=\varnothing $ and the result follows by Proposition \ref{corollario chiuso}.
\end{proof}

\begin{cor} \label{banale chiuso}
Let $G\subseteq S_k$ be a subgroup. Then
$\gr_{\mathbf{1}_G}(k,n)=\imm\left(\widehat{P}^{(n)}_{\mathbf{1}_G}\right)$.
\end{cor}

Notice that, in the trivial character case, we have the following commutative diagram of functions, where $\sim_G$ is the equivalence relation on $\PP(V)^{\times k}$ defined by setting $([u_1],\ldots,[u_k]) \sim_G ([v_1],\ldots,[v_k])$ if and only if $([v_1],\ldots,[v_k])=([u_{g^{-1}(1)}],\ldots,[u_{g^{-1}(k)}])$, for some $g\in G$, and $\pi_G$ is the canonical projection on the quotient.

\[\begin{tikzcd}
\PP(V)^{\times k} \arrow{r}{\mathrm{seg}^{k,n}} \arrow[swap]{d}{\pi_G} &  \PP(V^{\otimes k}) \arrow{d}{P^{(n)}_{\mathbf{1}_{G}}} \\
\faktor{\PP(V)^{\times k}}{\sim_G} \arrow{r}{\mathrm{seg}^{k,n}_G} & \PP(V^{\otimes k}_{\mathbf{1}_{G}})
\end{tikzcd}
\]
The function $\mathrm{seg}^{k,n}_G$ is the unique one such that $\widehat{P}^{(n)}_{\mathbf{1}_{G}}\circ \mathrm{seg}^{k,n}= \mathrm{seg}^{k,n}_G \circ \pi_G$. It is an injective function by \cite[Theorem 6.60]{Merris}; moreover $\imm(\mathrm{seg}^{k,n}_G)=\gr_{\mathbf{1}_G}(k,n)$.

Our next aim is to prove that the variety $\gr_{\mathbf{1}_{G}}(k,n)$ has the maximality property.
We define an idempotent function $p_{\mathbf{1}_{G}} : [n]^k \rightarrow B_{\mathbf{1}_{G}}(k,n)$ by setting $p_{\mathbf{1}_{G}}(x)=\overline{x}$, for all $x\in [n]^k$.

\begin{prop} \label{ordine preservato}
 The function $p_{\mathbf{1}_{G}}: ([n]^k,\leqslant) \rightarrow B_{\mathbf{1}_{G}}(k,n)$ is order preserving.
\end{prop}
\begin{proof}
  Let $x,y\in [n]^k$ be such that $x \leqslant y$, and $g,h \in G$ such that $p_{\mathbf{1}_{G}}(x)=g(x)$ and $p_{\mathbf{1}_{G}}(y)=h(y)$.
  Hence $p_{\mathbf{1}_{G}}(x) = g(x) \leqslant g(y)=gh^{-1}(p_{\mathbf{1}_{G}}(y))$, i.e. $p_{\mathbf{1}_{G}}(x) \preccurlyeq p_{\mathbf{1}_{G}}(y)$.
\end{proof}

\begin{cor} \label{corollario massimalita}
  The set $\gr_{\mathbf{1}_{G}}(k,n)$ has the maximality property. In particular, $\supp_{\mathbf{1}_G}[v]$ is a $\mathbf{1}_G$-matroid, for every $[v] \in \gr_{\mathbf{1}_{G}}(k,n)$.
\end{cor}
\begin{proof}
Let $\left[P^{(n)}_{\mathbf{1}_G}(v_1 \otimes .... \otimes v_k)\right] \in \gr_{1_G}(k,n)$. By Proposition \ref{massimalità segre}, $\supp_{\mathbf{1}_G}[v_1 \otimes .... \otimes v_k]$ has a unique maximum $m \in [n]^k$. Hence, by Proposition \ref{ordine preservato}, $p_{\mathbf{1}_{G}}(m)$ is the maximum of
$\supp_{\mathbf{1}_G}\left[P^{(n)}_{\mathbf{1}_G}(v_1 \otimes .... \otimes v_k)\right]$.
\end{proof}

Let $\rho : [n]^k \rightarrow \N \cup \{0\}$ be the rank function of $[n]^k$, i.e.
$\rho(x)=\sum_{i=1}^k(x_k-1)$, for all $x\in [n]^k$. The restriction of $\rho$ provides the rank function of the poset $B_{\mathbf{1}_{G}}(k,n)$, as the next result shows.

\begin{prop} \label{prop graduato}
  The poset $B_{\mathbf{1}_{G}}(k,n)$ is graded with rank function $\rho$.
\end{prop}
\begin{proof}
  The poset $B_{\mathbf{1}_{G}}(k,n)$ has maximum and minimum,  $(n,n,\ldots,n)$ and $(1,1,\ldots,1)$, respectively.
  Let $x \vartriangleleft y$ in $B_{\mathbf{1}_{G}}(k,n)$. Then $x \leqslant g(y)$ for some $g\in G$.
  Let $z \in [n]^k$ be such that $x<z<y$. By Proposition \ref{ordine preservato} we have that $x \preccurlyeq p_{\mathbf{1}_{G}}(z) \preccurlyeq y$. Since $\rho(z)=\rho(h(z))$, for all $h\in G$, we have $x \prec p_{\mathbf{1}_{G}}(z) \prec y$, a contradiction.
  From this fact we conclude that $x \vartriangleleft y$ in $B_{\mathbf{1}_{G}}(k,n)$ implies $\rho(y)=\rho(x)+1$.
\end{proof}

\begin{oss} \label{rango preservato}
Notice that Proposition \ref{prop graduato} follows directly by \cite[Lemma 7]{sagan} since $B_{\mathbf{1}_{G}}(k,n)$ is a homogeneous quotient of the poset $B_{\mathbf{1}_{\{e\}}}(k,n)$.
  Moreover, by Propositions \ref{ordine preservato} and \ref{prop graduato} it follows that $p_{1_G}$ is a morphism of graded posets.
\end{oss}

In the following example we depict the Hasse diagram of one of these posets.

\begin{ex} \label{esempio ciclico}
  Let $k=n=3$, $A_3$ be the alternating subgroup of $S_3$.
  The poset $B_{\mathbf{1}_{A_3}}(3,3)$ has the following Hasse diagram:

  
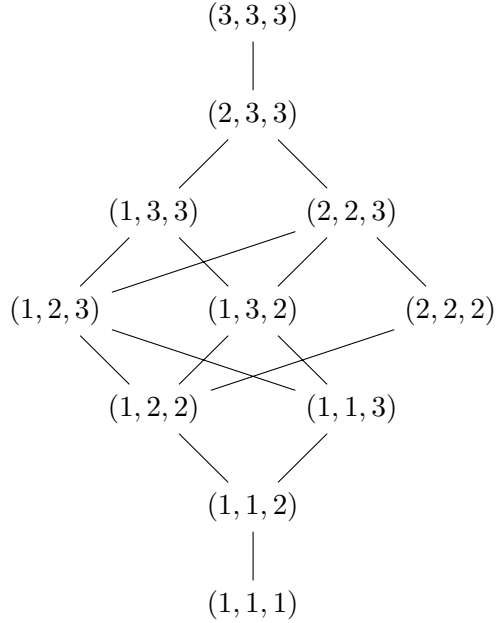
\begin{center}
 \begin{tikzpicture}[scale=1.3]
  \node (a1) at (2,6) {$(3,3,3)$};
  \node (b1) at (2,5) {$(2,3,3)$};
  \node (c1) at (1,4) {$(1,3,3)$};
  \node (c2) at (3,4) {$(2,2,3)$};
  \node (d1) at (0,3) {$(1,2,3)$};
  \node (d2) at (2,3) {$(1,3,2)$};
  \node (d3) at (4,3) {$(2,2,2)$};
  \node (e1) at (1,2) {$(1,2,2)$};
  \node (e2) at (3,2) {$(1,1,3)$};
  \node (f1) at (2,1) {$(1,1,2)$};
  \node (g1) at (2,0) {$(1,1,1)$};
  \draw (a1) -- (b1) (b1) -- (c1) (b1) -- (c2) (c1) -- (d1)(c1)--(d2) (c2)--(d1) (c2)--(d2) (c2)--(d3) (d1)--(e1)
  (d1)--(e2) (d2)--(e1) (d2)--(e2) (d3)--(e1) (e1)--(f1) (e2)--(f1) (f1)--(g1);
\end{tikzpicture}
  \end{center}
\end{ex}

By the weighted P\'olya enumeration theorem, one can deduce the rank-generating function of $B_{\mathbf{1}_{G}}(k,n)$.

\begin{prop} \label{polia}
The rank-generating function of the poset $B_{\mathbf{1}_G}(k,n)$ is
$$\sum \limits_{x \in B_{\mathbf{1}_{G}}(k,n)}q^{\rho(x)}=
\frac{1}{|G|}\sum \limits_{g\in G} \prod\limits_{i=1}^k\left([n]_{q^i}\right)^{c_i(g)},$$
where $[n]_q:=\sum \limits_{i=0}^{n-1}q^i$ is the $q$-analog of $n$ and $c_i(g)$ is the number of cycles of length $i \in [k]$ of
the permutation $g\in S_k$.
\end{prop}
\begin{proof}
The result follows by using the weighted P\'olya enumeration theorem with $[n]$ as set of colours and
weights $w(i)=i-1$, for all $i\in [n]$. Then the generating function of the colours is $[n]_q$ and $\rho(x)=\sum_{i=1}^k w(x_i)$, for all $x\in B_{\mathbf{1}_G}(k,n)$. Therefore the generating function of the number of colored arrangements by weight is the rank-generating function of $B_{\mathbf{1}_G}(k,n)$.
\end{proof}
\begin{cor} \label{specchio} Let $G \subseteq S_k$ be a group; then
the poset $B_{\mathbf{1}_G}(k,n)$ is rank-symmetric.
\end{cor}
\begin{proof}
Let $B(q):=
\frac{1}{|G|}\sum \limits_{g\in G} \prod\limits_{i=1}^k\left([n]_{q^i}\right)^{c_i(g)}$. Since $\deg(B(q))=k(n-1)$, the result can be deduced by the easy equality $q^{k(n-1)}B(q^{-1})=B(q)$.
\end{proof}

For the rest of this section, we focus on the case $G=S_k$. As for the Bruhat order on Grassmannian permutations, the poset $B_{\mathbf{1}_{S_k}}(k,n)$ is an induced subposet of $[n]^k$
(see \cite[Proposition 4.9]{BoloSenti} for a similar statement in the Grassmannian context). It is not difficult to prove that it is a distributive lattice. Then, a direct consequence of \cite[Theorem 3.7 and Proposition 4.2]{biorner} is that the intervals of $B_{\mathbf{1}_{S_k}}(k,n)$ are shellable.

\begin{oss}
  The poset $B_{\mathbf{1}_G}(k,n)$ could be not a lattice; see Example \ref{esempio ciclico}.
\end{oss}

In the last result of this section we provide an easy way to produce $\mathbf{1}_{S_k}$-matroids.

\begin{prop} \label{prop rappresentabili}
Let $[x,y]$ be an interval of $B_{\mathbf{1}_{S_k}}(k,n)$; then $[x,y]$ is a $\mathbf{1}_{S_k}$-matroid representable over $\C$.
\end{prop}
\begin{proof} We first prove that $[x,y]$ is representable over $\C$.
This follows by noting that, if $$[w]:=\left[\left(\sum\limits_{i_1=x_1}^{y_1}e_{i_1}\right)\otimes \ldots \otimes \left(\sum\limits_{i_k=x_k}^{y_k}e_{i_k}\right)\right] \in \seg(k,n),$$
then $\supp_{\mathbf{1}_{S_k}}\left[P^{(n)}_{\mathbf{1}_{S_k}}(w)\right]=[x,y]$. The inclusion $[x,y] \subseteq \supp_{\mathbf{1}_{S_k}}\left[P^{(n)}_{\mathbf{1}_{S_k}}(w)\right]$ is trivial. The other inclusion follows since $x_1\leqslant x_2 \leqslant \ldots \leqslant x_k$ and
 $y_1\leqslant y_2 \leqslant \ldots \leqslant y_k$, and then the following fact holds: if $1 \leq i<j \leq k$ and $a>b$ such that $x_i\leqslant a \leqslant y_i$, $x_j\leqslant b \leqslant y_j$, then $x_i\leqslant b \leqslant y_i$ and $x_j\leqslant a \leqslant y_j$.
 By Corollary \ref{corollario massimalita} we have that $[x,y]$ is
 a $\mathbf{1}_{S_k}$-matroid.
\end{proof}

\begin{oss}
Notice that a Bruhat interval $[x,y]$ of Grassmannian permutations is a matroid (see \cite[Section 2.3]{BoloSenti} and references therein). In general, any Bruhat interval in a parabolic quotient of a finite Coxeter group is a Coxeter matroid (see \cite{capelli}).
\end{oss}

The result of Proposition \ref{prop rappresentabili} is not true for the trivial character of an arbitrary group. In fact, consider  the poset of Example \ref{esempio ciclico}, and its interval
$[(1,1,3),(1,3,3)]$. Hence, if $\sigma=132 \in S_3$, we obtain
$$\left\{\sigma^*(x): x \in [(1,1,3),(1,3,3)]\right\}=\left\{(1,1,2),(1,2,3),(1,3,2),(1,2,2)\right\}.$$ This set has two maxima,
$(1,2,3)$ and $(1,3,2)$, so $[(1,1,3),(1,3,3)]$ is not a $\chi$-matroid; in particular, it is not representable over $\C$, because $\gr_{\mathbf{1}_{A_3}(3,3)}$ has the maximality property by Corollary \ref{corollario massimalita}.

\section{The geometry of $\gr_{\mathbf{1}_G}(k,n)$} \label{sezione ciao}

In this section we use ideas from \cite{BoloSenti} and the combinatorial tools of Section \ref{sezione triviale} to realize an incidence stratification of the variety $\gr_{\mathbf{1}_{G}}(k,n)$. This stratification provides a set of generators for the Chow $\Q$-vector space of $\gr_{\mathbf{1}_{G}}(k,n)$, as we prove in Theorem \ref{teorema ciao}.
We consider the incidence group over $\C$ of the poset $B_{\mathbf{1}_{G}}(k,n)$   acting on the projective space $\PP\left(V^{\otimes k}_{\mathbf{1}_G}\right)$.
Then, by \cite[Theorem 5.1]{BoloSenti} the orbits of this action are in bijection with the order ideals of $B_{\mathbf{1}_{G}}(k,n)$.
Given an order ideal $I\in \mathcal{J}\left(B_{\mathbf{1}_{G}}(k,n)\right)$ let $C_I$ be the corresponding orbit and define
$$ C^{\mathbf{1}_{G}}_I:= C_I \cap \gr_{\mathbf{1}_G}(k,n).$$
\begin{prop} \label{teorema ideali ordine principali}
  We have that $C^{\mathbf{1}_{G}}_I \neq \varnothing$ if and only if $I$ is a principal order ideal. Hence as posets $$\left\{I\in \mathcal{J}\left(B_{\mathbf{1}_{G}}(k,n)\right):C^{\mathbf{1}_{G}}_I \neq \varnothing \right\} \simeq B_{\mathbf{1}_{G}}(k,n).$$
\end{prop}
\begin{proof}
  Let $I$ be principal with $m:=\max(I)$. Then $\left[P^{(n)}_{\mathbf{1}_G}(e_m)\right] \in C_I \cap \gr_{\mathbf{1}_G}(k,n)$. On the other hand, let $[v] \in C_I \cap \gr_{\mathbf{1}_G}(k,n)$ for some $v \in V^{\otimes k}_{\mathbf{1}_G}$. By
  Corollary \ref{corollario massimalita}, we have that $\max(I)=\max(\supp_{\mathbf{1}_{G}}[v])=\{m\}$
  for some $m \in B_{\mathbf{1}_{G}}(k,n)$. The rest of the statement is an immediate consequence.
\end{proof}

For $x\in B_{\mathbf{1}_{G}}(k,n)$, we  set $C_x := C^{\mathbf{1}_{G}}_{x^{\downarrow}}$. By Proposition \ref{teorema ideali ordine principali}, it follows that
$$\gr_{\mathbf{1}_G}(k,n) = \biguplus_{x \in B_{\mathbf{1}_{G}}(k,n)}C_x.$$

Let $x\in [n]^k$ and $A\in \mat_{n,k}(\C)$ be the matrix of parameters (see Theorem \ref{equazioni}). We let $A^x\in \mat_{n,k}(\C)$ be the matrix whose entries are
$A^x_{i,j}=A_{i,j}$ if $i \leqslant x_j$, and $A^x_{i,j}=0$ otherwise.
We now provide a system of parametric equations for the algebraic set $\overline{C}_{x^\downarrow} \cap \gr_{\mathbf{1}_G}(k,n)$.

\begin{thm} \label{parametriche sciubert}
Let $x\in B_{\mathbf{1}_G}(k,n)$. Then $ \overline{C}_{x^\downarrow} \cap \gr_{\mathbf{1}_G}(k,n)$ is described by the parametric equations
 $$\left\{z_y =|G_y|^{-1}(\mathbf{1}_G)_{y,(1,\ldots,k)}(A^{x}) : y \in x^\downarrow\right\} \cup \left\{z_y=0: y\not \in x^{\downarrow}\right\}.$$
\end{thm}

\begin{proof}
First consider the case $G=\{e\}$, the trivial group. Let $x\in [n]^k$ and
$[u]\in\overline{C}_{x^\downarrow} \cap \seg(k,n)$; hence, by Proposition \ref{teorema ideali ordine principali}, $z:=\max (\supp_{\mathbf{1}_{G}}[u])\leqslant x$ and then
$v=\sum_{y\leqslant x} (A^x)^{\otimes k}_{y,(1,\ldots,k)}e_y$ is a representative of $[u]$, for the choice $a_{ij}=0$ whenever $i>z_i$,
for all $i\in [n]$ and for all $j\in [k]$.

Let $G\neq \{e\}$, $x\in B_{\mathbf{1}_G}(k,n)\subseteq [n]^k$ and $[u]\in \overline{C}_{x^{\downarrow G}} \cap \gr_{\mathbf{1}_G}(k,n)$, $x^{\downarrow G} \in \mathcal{J}(B_{\mathbf{1}_G}(k,n))$. Then, by Proposition \ref{ordine preservato}, a representative of $[u]$ is
$w=\frac{1}{|G|}\sum\limits_{g\in G}gv$, for some $v \in \overline{C}_{x^\downarrow} \cap \seg(k,n)$,
where  $x^\downarrow \in \mathcal{J}([n]^k)$ and $\overline{x}=x$.

We claim that, if $x\in [n]^k$, then $p_{\mathbf{1}_G}(x^\downarrow)=\left(p_{\mathbf{1}_G}(x)\right)^\downarrow$. First we prove that $p_{\mathbf{1}_G}(x^\downarrow) \subseteq \left(p_{\mathbf{1}_G}(x)\right)^\downarrow$. Let $y \in p_{\mathbf{1}_G}(x^\downarrow)$; then $y=p_{\mathbf{1}_G}(z)$ for some $z\leqslant x$. By Proposition \ref{ordine preservato} the function $p_{\mathbf{1}_G}$ is order preserving, hence
$y=p_{\mathbf{1}_G}(z)\preccurlyeq p_{\mathbf{1}_G}(x)$, i.e. $y\in \left(p_{\mathbf{1}_G}(x)\right)^\downarrow$.
We prove now that
$\left(p_{\mathbf{1}_G}(x)\right)^\downarrow \subseteq  p_{\mathbf{1}_G}(x^\downarrow)$. Let
    $y\in \left(p_{\mathbf{1}_G}(x)\right)^\downarrow $; then $y \preccurlyeq p_{\mathbf{1}_G}(x)$, i.e. $y  \leq g(x)$, for some $g\in G$.
    Therefore $z:=g^{-1}(y) \leqslant x$.
    Hence $z \in x^\downarrow$ and the result follows because $y=p_{\mathbf{1}_G}(z)$.

Our claim implies that, as sets, $$\mathcal{J}([n]^k) \ni x^{\downarrow} \simeq \biguplus\limits_{z \preccurlyeq x} \bigcup \limits_{\{g\in G:g(z)\leqslant x\}}\left\{(z,g(z))\right\}.$$
Therefore, by the previous case and our claim,
\begin{eqnarray*} w&=&\frac{1}{|G|}\sum\limits_{h\in G}h\sum\limits_{y\leqslant x} (A^x)^{\otimes k}_{y,(1,\ldots,k)}e_y \\
&=&  \frac{1}{|G|}\sum\limits_{h\in G}h \sum \limits_{y\preccurlyeq x}\sum \limits_{\substack{u\in O_y\\ u\leq x}}(A^x)^{\otimes k}_{u,(1,\ldots,k)}e_{u}\\
&=&  \frac{1}{|G|}\sum \limits_{y\preccurlyeq x}\left(\sum \limits_{u\in O_y}(A^x)^{\otimes k}_{u,(1,\ldots,k)}\right)\sum_{h\in G}e_{h(y)}\\
&=&  \sum \limits_{y\preccurlyeq x}|G_y|^{-1}(\mathbf{1}_G)_{y,(1,\ldots,k)}(A^x)P^{(n)}_{\mathbf{1}_G}(e_y).
\end{eqnarray*}
\end{proof}

\begin{cor} \label{corollario Schubert irriducibili}
Let $x\in B_{\mathbf{1}_{G}}(k,n)$. Then the projective algebraic set $\overline{C}_{x^\downarrow} \cap \gr_{\mathbf{1}_G}(k,n)$ is irreducible. In particular $$\overline{C}_x=\overline{C}_{x^\downarrow} \cap \gr_{\mathbf{1}_G}(k,n),$$ where $\overline{C}_x$ and $\overline{C}_{x^\downarrow}$ are the Zariski closures in $\gr_{\mathbf{1}_G}(k,n)$ and $\PP\left(V^{\otimes k}_{\mathbf{1}_G}\right)$, respectively.
\end{cor}
\begin{proof}
As in Proposition \ref{teorema irriducibile}, Theorem \ref{parametriche sciubert} implies that $\overline{C}_x$ is irreducible. The equality can be deduced by noting that $C_{x^\downarrow} \cap \gr_{\mathbf{1}_G}(k,n)$ is an open set of the projective variety
$\overline{C}_{x^\downarrow} \cap \gr_{\mathbf{1}_G}(k,n)$.
\end{proof}

We have proved in the previous section that the poset $B_{\mathbf{1}_G}(k,n)$ is graded; in the last result of this section we interpret geometrically its rank function.
Moreover we provide a set of generators for the Chow vector space $A_*\left(\gr_{\mathbf{1}_G}(k,n);\Q\right)$ of $\gr_{\mathbf{1}_G}(k,n)$ over $\Q$. Recall from Proposition \ref{polia} that $c_i(g)$ is the number of cycles of length $i \in [k]$ of
the permutation $g\in S_k$, and denote by $\mathrm{HP}_G$ the Hilbert-Poincar\'e polynomial of $A_*\left(\gr_{\mathbf{1}_G}(k,n);\Q\right)$.

\begin{thm} \label{teorema ciao}
    Let $G\subseteq S_k$ be a group and $x\in B_{\mathbf{1}_G}(k,n)$. Then:
    \begin{enumerate}
        \item $\dim(\overline{C}_x)=\rho(x)$; in particular $\dim(\gr_{\mathbf{1}_G}(k,n))=k(n-1)$.
        \item $\left\{[\overline{C}_x]:x \in B_{\mathbf{1}_{G}}(k,n)\right\}$ is a set of generators of the Chow vector space over
$\Q$ of $\gr_{\mathbf{1}_G}(k,n)$; in particular
$$
\mathrm{HP}_G(q)\leq\frac{1}{|G|}\sum \limits_{g\in G} \prod\limits_{i=1}^k\left([n]_{q^i}\right)^{c_i(g)},$$
    \end{enumerate} where the partial order on polynomials is taken coefficient-wise.
\end{thm}
\begin{proof}
Let $G=\{e\}$. It is easy to verify that $\overline{C}_x \simeq \prod_{i=1}^k \PP(\mathbb{C}^{x_i})$ as projective varieties.
Hence $\dim(\overline{C}_x)=\sum\limits_{i\in [k]}(x_i-1)=\rho(x)$.
In this case the second statement is clear. In fact $\seg(k,n) \simeq \prod_{i=1}^k\PP(\C^n)$ as algebraic varieties, and a basis for the Chow vector space of the second is $\left\{\prod_{i=1}^k [\overline{C}_{x_i}]: x \in [n]^k\right\}$; this basis corresponds to $\left\{[\overline{C}_x]:x \in [n]^k\right\}$ by using the exterior product isomorphism (see \cite[Example 8.3.7]{Fulton}).

Assume now $\{e\} \subsetneq G$.
    The morphism $\widehat{P}^{(n)}_{\mathbf{1}_G}: \seg(k,n) \rightarrow \gr_{\mathbf{1}_G}(k,n)$ has finite fibers by \cite[Theorem 6.60]{Merris}. Since it is a surjective morphism between projective varieties, for every subvariety $X\subseteq \seg(k,n)$ of dimension $d$ we have that $\widehat{P}^{(n)}_{\mathbf{1}_G}(X)$ is a $d$-dimensional subvariety of $\gr_{\mathbf{1}_G}(k,n)$.
    As we stated in the proof of Theorem \ref{parametriche sciubert},  $p_{\mathbf{1}_G}(x^\downarrow)=\left(p_{\mathbf{1}_G}(x)\right)^\downarrow$; hence
  $\widehat{P}^{(n)}_{\mathbf{1}_G}(\overline{C}_{x^{\downarrow}})=\overline{C}_{x^{\downarrow G}}\subseteq \PP(V^{\otimes k}_{\mathbf{1}_G})$. It follows that
    $$\widehat{P}^{(n)}_{\mathbf{1}_G}\left(\overline{C}_{x^{\downarrow}} \cap \seg(k,n)\right) \subseteq \widehat{P}^{(n)}_{\mathbf{1}_G}(\overline{C}_{x^{\downarrow }}) \cap \gr_{\mathbf{1}_G}(k,n)=\overline{C}_{x^{\downarrow G}} \cap \gr_{\mathbf{1}_G}(k,n).$$ Therefore, by Corollary \ref{corollario Schubert irriducibili} and the fact that $\widehat{P}^{(n)}_{\mathbf{1}_G}$ preserves the dimensions, we obtain $\widehat{P}^{(n)}_{\mathbf{1}_G}\left(\overline{C}_{x^{\downarrow}} \cap \seg(k,n)\right)=\overline{C}_{x^{\downarrow G}} \cap \gr_{\mathbf{1}_G}(k,n)$. This proves the first statement.

    We prove now the second statement. The morphism $\widehat{P}^{(n)}_{\mathbf{1}_G}$ preserves the dimensions and then the push-forward $$(\widehat{P}^{(n)}_{\mathbf{1}_G})_* : A_*(\seg(k,n);\Q) \rightarrow A_*(\gr_{\mathbf{1}_G}(k,n);\Q)$$ is surjective; so the result follows by the $G=\{e\}$ case and the fact that $\widehat{P}^{(n)}_{\mathbf{1}_G}(\overline{C}_{x^\downarrow} \cap \seg(k,n))=\overline{C}_{x^{\downarrow G}} \cap\gr_{\mathbf{1}_G}(k,n)$. The inequality is a consequence of Proposition \ref{polia}.
\end{proof}
For the Segre variety $\seg(k,n)$, the strata of the incidence stratification defined above are smooth. Hence, the incidence stratification of $\seg(k,n)$ we have realized in this section, is a {\em Seshadri stratification}, in the sense of \cite[Definition 2.1]{chirivi}. Now we provide an example of incidence stratification which is not Seshadri.
\begin{ex} \label{osservazione i}
  Let $k=2$, $n=3$ and $G=S_2$. By using Theorem \ref{parametriche sciubert} and Corollary \ref{corollario Schubert irriducibili} for $x=(3,3)$, and the implicitization method, we obtain that the Chow variety
  $\gr_{\mathbf{1}_{S_2}}(2,3) \subseteq \PP\left(V_{\mathbf{1}_G}^{\otimes 2}\right)$ is defined by the equation
  $$x_{13}^2x_{22} - x_{12}x_{13}x_{23} + x_{11}x_{23}^2 + x_{12}^2x_{33} - 4x_{11}x_{22}x_{33}=0. $$
  Notice that this is a Brill's equation, see \cite[Section 5.4]{briand}.

  The variety $\overline{C}_{(2,3)}$ has dimension $\rho((2,3))=3$
  and $\overline{C}_{(2,2)}$ is its singular locus, whose  defining ideal is $(x_{13},x_{23},x_{33})+I(\gr_{\mathbf{1}_{S_2}}(2,3) )=(x_{13},x_{23},x_{33})$, and it has dimension $2$.
  Hence $\overline{C}_{(2,3)}$ is not smooth in codimension one, i.e. the incidence stratification $\{C_x: x \in B_{\mathbf{1}_{S_2}}(2,3)\}$ of $\gr_{\mathbf{1}_{S_2}}(2,3)$  is not a Seshadri stratification.
\end{ex}

\section{Two conjectures and three problems} \label{sezione problematica}
It is well-known that Schubert varieties of the Grassmannian are Cohen-Macaulay, see \cite{chirivi0}, \cite{deconcini} and \cite{Hochster} and reference therein. Supported by several computations with {\em Macaulay2}, we formulate a conjecture about the geometry of the incidence strata of $\gr_{\mathbf{1}_G}(k,n)$ defined in the previous section.

\begin{conge} \label{conge CM}
Let $G\subseteq S_k$ be a group and $x \in B_{\mathbf{1}_G}(k,n)$; then
$\overline{C}_x$ is Cohen-Macaulay.
\end{conge}

If the techniques based on flat deformations of $LS$ algebras would be available in our context (see \cite{chirivi0}), the previous conjecture could be related to the following one, which has an independent interest.

\begin{conge} \label{conge shell}
Let $G\subseteq S_k$ be a group.
Then the order complex of the poset $B_{\mathbf{1}_G}(k,n)$ is shellable.
\end{conge}

In Section \ref{sezione unidimensionale} we defined a poset $B_{\chi}(k,n)$ for one-dimensional characters $\chi$. In this case, what we know is that
if $\chi$ is the alternating character of $S_k$, $B_{\chi}(k,n)$ is the Bruhat order on Grassmannian permutations, hence it has shellable intervals (see \cite[Chapter 2]{BB}).

Our conjecture does not extend to other characters, because, for instance, if $G=\{e,4321\}\subseteq S_4$ and $\chi(4321)=-1$,
then the order complex of $B_\chi(4,3)$ is not shellable. This leads to the following problem.

\begin{prob}
For which one-dimensional characters $\chi$ have we that $B_{\chi}(k,n)$ is graded? What is the topology of its intervals?
\end{prob}

In Proposition \ref{corollario chiuso}, we provide a condition ensuring that $\imm(\,\widehat{f}\,)$ is Zariski closed; this condition is satisfied in the case of trivial characters (Corollary \ref{banale chiuso}). Then the following problem seems to be natural.

\begin{prob}
For which $f \in \End\left(V^{\otimes k}\right)$ have we that $\imm(\,\widehat{f}\,)$ is Zariski closed, i.e. $\gr_f(k,n)=\imm(\,\widehat{f}\,)$?
\end{prob}

Along the paper, we define $\chi$-matroids for a one-dimensional character $\chi$ of a group $G \subseteq S_k$ (see Definition \ref{chi matroidi}). We believe that this notion has an independent interest.

\begin{prob}
Find crypthomorphic definitions of $\chi$-matroids and explore their combinatorial properties.
\end{prob}

\section{Acknowledgements}
We thank Christophe Reutenauer for inspiring the content of Example \ref{lyndon}. The first author is grateful to Dipartimento di Matematica of Politecnico di Milano for its hospitality and financial support.

Both authors are grateful to the amenity of Robecco sul Naviglio for inspiring some ground ideas of the present work.


\begin{thebibliography}{9}

\bibitem{biorner}
A. Bj\"{o}rner, {\em Shellable and Cohen-Macaulay partially ordered sets},  Transactions of the American Mathematical Society 260.1, 159-183 (1980).

\bibitem{BB}
A. Bj\"{o}rner and F. Brenti, {\em Combinatorics of Coxeter Groups},
Graduate Texts in Mathematics, 231, Springer-Verlag, New York, 2005.

\bibitem{BoloSenti}
D. Bolognini and P. Sentinelli, {\em P-flag spaces and incidence stratifications},  Selecta Mathematica, New Series 27, 72 (2021).

\bibitem{briand}
E. Briand, {\em Covariants vanishing on totally decomposable forms},  Liaison, Schottky problem and invariant theory, Birkhäuser, 237-256, 2010.

\bibitem{coxeter matroids}
A. Borovik, I. M. Gelfand  and N. White, {\em Coxeter matroids},
Birkh\"{a}user, Progress in Mathematics, 216, 2003.

\bibitem{capelli}
F. Caselli, M. D'Adderio and M. Marietti, {\em Weak generalized lifting property, Bruhat intervals and Coxeter matroids},  International Mathematical Research Notices 2021.3, 1678-1698 (2021).

\bibitem{scarabotti}
T. Ceccherini-Silberstein, F. Scarabotti and F. Tolli,
{\em Representation theory of the symmetric groups: the Okounkov-Vershik approach, character formulas, and partition algebras},
Cambridge University Press, 2010.

\bibitem{chirivi0}
R. Chirivì, {\em LS algebras and application to Schubert varieties},  Transformation groups 5.3, 245-264 (2000).

\bibitem{chirivi}
R. Chirivì, X. Fang, and P. Littelmann, {\em Seshadri stratifications and standard monomial theory}, Inventiones mathematicae (2023): 1-84.

\bibitem{cox}
D. Cox, J. Little, and D. O'Shea, {\em Ideals, varieties, and algorithms: an introduction to computational algebraic geometry and commutative algebra}, Springer Science \& Business Media, 2013.


\bibitem{deconcini}
C. De Concini and V. Lakshmibai, {\em Arithmetic Cohen-Macaulayness and arithmetic normality for Schubert varieties}, American Journal of Mathematics 103.5, 835-850 (1981).

\bibitem{portoghesi}
R. Fernandes and H. F. da Cruz, {\em The Multilinear Algebra of José Dias da Silva and the Portuguese school of mathematics},  Linear algebra and its applications 436.6, 1545-1561 (2012).

\bibitem{fink}
A. Fink and D. E. Speyer, {\em K-classes for matroids and equivariant localization}, Duke Math. J. 161.14, 2699-2723 (2012).

\bibitem{Fulton}
W. Fulton, {\em Intersection theory},  Princeton University Press, 2016.


\bibitem{Gelfand}
I. M. Gelfand, R. M. Goresky, R. D. MacPherson and V. V. Serganova, {\em Combinatorial geometries, convex polyhedra, and Schubert cells}, Advances in Math. 63.3, 301-316 (1987).

\bibitem{discriminant}
I. M. Gelfand, M. Kapranov, and A. Zelevinsky, {\em Discriminants, resultants, and multidimensional determinants}, Birkhäuser, 1994.

\bibitem{M2}
D. R. Grayson and M. E. Stillman, {\em Macaulay2, a software system for research in algebraic geometry}, {\tt http://www.math.uiuc.edu/Macaulay2/}.


\bibitem{sagan}
J. Hallam and B. Sagan, {\em Factoring the characteristic polynomial of a lattice}, Journal of Combinatorial Theory, Series A 136, 39-63 (2015).

\bibitem{Hochster}
M. Hochster, {\em Grassmannians and their Schubert subvarieties are arithmetically Cohen-Macaulay}, Journal of Algebra,
Volume 25, Issue 1, 40-57 (1973).

\bibitem{Landsberg}
J. M. Landsberg,
{\em Tensors: geometry and applications},
Vol. 128, American Mathematical Society, 2011.

\bibitem{lang}
S. Lang,
{\em Algebra},
Springer, 2002.

\bibitem{induced operators}
C. K. Li and A. Zaharia, {\em Induced operators on symmetry classes of tensors}, Transactions of the American Mathematical Society, 354.2, 807-836 (2002).

\bibitem{marcus}
M. Marcus and H. Minc, {\em Generalized matrix functions}, Transactions of the American Mathematical Society 116, 316-329  (1965).

\bibitem{Merris}
R. Merris,
{\em Multilinear algebra}, CRC Press, 1997.


\bibitem{sagemath}
{The Sage Developers}, \emph{{S}age{M}ath, the {S}age {M}athematics {S}oftware {S}ystem ({V}ersion 7.3)}, {\tt https://www.sagemath.org}.

\bibitem{Shafarevich}
I. R. Shafarevich, {\em Basic Algebraic Geometry}, Vol. 1, Springer Verlag, 1994.


\bibitem{Stanley}
R. P. Stanley, {\em Enumerative Combinatorics}, Vol. 1, Wadsworth and
Brooks/Cole, Monterey, CA, 1986.


\bibitem{Stenley}
R. P. Stanley, {\em Enumerative Combinatorics }, Vol. 2,
Cambridge Studies in Advanced Mathematics, 62,
Cambridge University Press, 1999.

\bibitem{stembridge}
J. R. Stembridge, {\em Some conjectures for immanants},  Canadian Journal of Mathematics 44.5, 1079-1099 (1992).

\end{thebibliography}
\end{document}